\documentclass[finalsize,finaltitle,draft]{zaa}

\usepackage{enumitem}
\usepackage[pdfstartview=FitV,bookmarks=false]{hyperref} 
\usepackage{multicol}

\DeclareMathOperator{\Img}{Im}
\DeclareMathOperator{\Log}{Log}

\newcommand{\reg}[1]{\mathcal{R}^{#1}}
\newcommand{\cnt}{\mathrm{C}}
\newcommand{\creg}[1]{\mathcal{R}_{c}^{#1}}
\newcommand{\crd}[1]{\mathrm{C}_{\mathrm{rd}}^{#1}}
\newcommand{\C}{\mathbb{C}}
\newcommand{\N}{\mathbb{N}}
\newcommand{\R}{\mathbb{R}}
\newcommand{\T}{\mathbb{T}}
\newcommand{\Z}{\mathbb{Z}}
\newcommand{\qZ}{\overline{q^{\mathbb{Z}}}}

\newcommand{\cf}{A}
\newcommand{\dly}{\alpha}
\newcommand{\ftrm}{f}
\newcommand{\initval}{x_{0}}
\newcommand{\initfun}{\varphi}
\newcommand{\fun}{\mathcal{X}}

\begin{document}

\title{On Different Types of Stability\\ for Linear Delay Dynamic Equations}

\runtitle{Stability for Linear Delay Dynamic Equations}

\author{Elena Braverman and Ba\c{s}ak Karpuz}
\runauthor{E. Braverman and B. Karpuz}

\address{E. Braverman: Department of Mathematics and Statistics, University of Calgary, 2500 University Dr. NW, Calgary, AB T2N 1N4, Canada.
\email{maelena@ucalgary.ca}}
\address{B. Karpuz: Department of Mathematics, Faculty of Science, T{\i}naztepe Campus, Dokuz Eyl\"{u}l University, Buca, 35160 \.{I}zmir, Turkey.
\email{bkarpuz@gmail.com}}


\abstract{
We provide explicit conditions for uniform stability, global asymptotic stability
and uniform exponential stability for dynamic equations with a single delay and a nonnegative
coefficient.
Some examples on nonstandard time scales are also given to show applicability and sharpness of the new
results.
}



\primclass{34K20}
\secclasses{34N05}

\received{May 2, 2016}
\revised{September 26, 2016}
\logo{36}{2017}{3}{343}
\doi{3445}

\maketitle

\section{Introduction}

Different types of stability for linear delay differential and difference equations, even with a single delay,
continue to attract attention, see, for example, the recent papers \cite{MR3396349,MR3418565} and references therein.
For linear delay equations, the stability type has been connected to properties of the kernels of
solution representations \cite{MR2266563,MR2795517,MR3418565}.

The purpose of the present paper is two-fold:
\begin{enumerate}[label={\arabic*.},leftmargin={*},ref={\arabic*}]
\item to unify the results connecting different types of
    stability with estimates of the fundamental solutions
    for delay differential and difference equations,
    and extend them to equations on other types of time scales;
\item to outline the difference between discrete and continuous time scales:
    the two parts of our main results coincide for differential equations
    but have a meaningful difference for difference equations; for other time scales,
    the two conditions coincide at the right-dense points and differ at the right-scattered ones.
\end{enumerate}
In this paper, we study uniform stability, global asymptotic stability and uniform exponential stability
of solutions to the delay dynamic equation
\begin{equation}
x^{\Delta}(t)+\cf(t)x(\dly(t))=0\quad\text{for}\ t\in[t_{0},\infty)_{\T},\label{introeq1}
\end{equation}
where $\T$ is a time scale unbounded above, $\cf\in\crd{}([t_{0},\infty)_{\T},\R_{0}^{+})$
and $\dly\in\crd{}([t_{0},\infty)_{\T},\T)$ with $\lim_{t\to\infty}\dly(t)=\infty$.
Here, the notations $\R^{+}:=(0,\infty)$ and $\R_{0}^{+}:=[0,\infty)$ are used.
Further, we will be assuming one of the following properties for the delay function $\dly$:
\begin{enumerate}[label={(A\arabic*)},leftmargin={*},ref={(A\arabic*)}]
\item\label{a1} $\dly(t)\leq{}t$ for all $t\in[t_{0},\infty)_{\T}$;
\item\label{a2} $\dly(\sigma(t))\leq{}t$ for all $t\in[t_{0},\infty)_{\T}$.
\end{enumerate}
To use in the sequel, we need to define the least value of the delayed argument
\begin{equation}
\dly_{\ast}(t):=\inf\{\dly(\eta):\ \eta\in[t,\infty)_{\T}\}
\quad\text{for}\ t\in[t_{0},\infty)_{\T}.\notag
\end{equation}
Note here that $\dly_{\ast}(t)>-\infty$ for any $t\in[t_{0},\infty)_{\T}$
since $\lim_{t\to\infty}\dly(t)=\infty$ yields that there exists $t_{1}\in[t_{0},\infty)_{\T}$ such that
$\dly(t)\geq{}t_{0}$ for all $t\in[t_{1},\infty)_{\T}$ and $\inf_{\eta\in[t_{0},t_{1}]_{\T}}\dly_{\ast}(\eta)$
is finite by \cite[Theorem~1.60\,(ii) and Theorem~1.65]{MR1843232}.
Clearly, $\dly_{\ast}$ is a nondecreasing rd-continuous function on $[t_{0},\infty)_{\T}$.
Further, for monotone nondecreasing $\dly$, we have $\dly_{\ast}=\dly$ on $[t_{0},\infty)_{\T}$.
We also define
\begin{equation}
\dly_{-1}(t):=\sup\{\eta\in[t_{0},\infty)_{\T}:\ \dly_{\ast}(\eta)\leq{}t\}
\quad\text{for}\ t\in[t_{0},\infty)_{\T}.\notag
\end{equation}
It is easy to see that for each $s\in[t_{0},\infty)_{\T}$, we have $\dly_{\ast}(t)\geq{}s$ for all $t\in[\dly_{-1}(s),\infty)_{\T}$.

If the delay $\dly$ is strict, i.e., $\dly(t)<t$ for all $t\in[t_{0},\infty)_{\T}$, then \ref{a2} holds.
In particular, \ref{a1} and \ref{a2} are the same for $\T=\R$ since $\sigma(t)=t$ for all $t\in\R$,
while \ref{a1} is weaker than \ref{a2} for $\T=\Z$ since \ref{a2} means $\dly(t)\leq{}t-1$ for all $t\in\Z$.

The so-called Hilger-derivative $x^{\Delta}$ in \eqref{introeq1} turns out to be the usual derivative $x^{\prime}$ when $\T=\R$, and
the forward difference operator $\Delta$ when $\T=\Z$, i.e., $\Delta{}x(t)=x(t+1)-x(t)$ for $t\in\Z$.
Hence, our study here will unify some of the fundamental stability results for delay differential equations
\begin{equation}
x^{\prime}(t)+\cf(t)x(\dly(t))=0\quad\text{for}\ t\in[t_{0},\infty)_{\R} \label{introeq2}
\end{equation}
and  delay difference equations
\begin{equation}
\Delta{}x(t)+\cf(t)x(\dly(t))=0\quad\text{for}\ t\in[t_{0},\infty)_{\Z}.\label{introeq3}
\end{equation}
As presented in \cite[Examples~1.38{ -- }1.40]{MR1843232},
there exist some phenomena in real world applications which cannot be described by only either continuous or discrete models.
The present paper aims to extend the classical stability tests for more general type of equations called
dynamic equations.

Stability and asymptotic stability of the differential equation \eqref{introeq2} have been studied in \cite{MR1130462,MR0695252,MR1233677,MR0891356},
and the exponential stability of \eqref{introeq2} in \cite{MR2266563,MR2185238}.
Some results on stability and asymptotic stability of \eqref{introeq3} can be found in
\cite{MR1403453,MR1350434,MR1636333,MR2814561,MR3418565,MR1305480,MR1103855,MR1666138,MR1340734},
and some results on the exponential stability in \cite{MR2307338,MR2519555,MR3418565}.

We mention here \cite{MR2535793,MR2651874,MR2352908,MR2437887,MR3352729,MR2335381} and \cite{MR2640308,MR2116401,MR2510663,MR2313706,MR2464521,MR2197128,MR1974425,MR2292449},
which deal with asymptotic stability and exponential stability of the dynamic equation \eqref{introeq1}, respectively.
Further, we refer the readers to the papers \cite{MR1989020,MR1961592,MR1989027,MR3250516} for a discussion on stability and the time scale exponential function.

Let us proceed with definitions of the solution and the stability of \eqref{introeq1}.

\begin{definition}[Solution]\label{dfnsln}
A function $x:[\dly_{\ast}(t_{0}),\infty)_{\T}\to\R$,
which is rd-continuous on $[\dly_{\ast}(t_{0}),t_{0}]_{\T}$ and $\Delta$-differentiable on $[t_{0},\infty)_{\T}$
with an rd-conti\-nuous derivative,
is called a \textbf{solution} of \eqref{introeq1}
provided that it satisfies the equality \eqref{introeq1} identically on $[t_{0},\infty)_{\T}$.
\end{definition}

\begin{definition}[Uniform Stability]\label{dfnus}
The trivial solution of \eqref{introeq1} is said to be \textbf{uniformly stable}
if for any $\varepsilon\in\R^{+}$, there exists $\delta\in\R^{+}$ such that for any $s\in[t_{0},\infty)_{\T}$,
any solution $x$ of the initial value problem
\begin{equation}
\begin{cases}
x^{\Delta}(t)+\cf(t)x(\dly(t))=0\quad\text{for}\ t\in[s,\infty)_{\T}\\
x(s)=\initval\quad\text{and}\quad{}x(t)=\initfun(t)\quad\text{for}\ t\in[\dly_{\ast}(s),s)_{\T},
\end{cases}\label{predf2eq1}
\end{equation}
with $|\initval|+\sup_{\eta\in{}[\dly_{\ast}(s),s)_{\T}}|\initfun(\eta)|<\delta$ satisfies $|x(t)|<\varepsilon$ for all $t\in[s,\infty)_{\T}$.
\end{definition}


\begin{definition}[Global Attractivity]\label{dfnga}
The trivial solution of \eqref{introeq1} is said to be \textbf{globally attracting}
if for any $s\in[t_{0},\infty)_{\T}$, any solution $x$ of the initial value problem \eqref{predf2eq1} satisfies $\lim_{t\to\infty}x(t)=0$.
\end{definition}

\begin{definition}[Global Asymptotic Stability]\label{dfngas}
The trivial solution of \eqref{introeq1} is said to be \textbf{globally asymptotically stable} if it is uniformly stable and globally attracting.
\end{definition}

\begin{definition}[Uniform Exponential Stability]\label{dfnues}
The trivial solution of \eqref{introeq1} is said to be \textbf{uniformly exponentially stable}
if there exist constants $M,\lambda\in\R^{+}$ such that for any $s\in[t_{0},\infty)_{\T}$,
any solution $x$ of the initial value problem \eqref{predf2eq1} satisfies
\begin{equation}
|x(t)|\leq{}M\mathrm{e}_{\ominus\lambda}(t,s)\bigg(|\initval|+\sup_{\eta\in{}[\dly_{\ast}(s),s)_{\T}}|\initfun(\eta)|\bigg)\quad\text{for all}\ t\in[s,\infty)_{\T}.\notag
\end{equation}
\end{definition}

It is shown in \cite[Theorem~2.1]{MR2927064} that for time scales with bounded graininess (such as $\T=\R$ and $\T=\Z$),
the term $\mathrm{e}_{\ominus\lambda}(t,s)$, where $\lambda\in\R^{+}$, in Definition~\ref{dfnues} can be replaced by
$\mathrm{e}_{-\lambda}(t,s)$, where $\lambda\in\R^{+}$ and $1-\lambda\mu(t)>0$ for all $t\in\T$ (see \cite[Theorem~2.1]{MR2927064}).

Obviously, one has the following implication chart (see \cite[Corollary~5.7]{MR3250516} or Corollary~\ref{alcrl1}):
\begin{center}
\begin{tabular}{c}
  Uniform\\
  Exponential\\
  Stability
\end{tabular}
$\implies$
\begin{tabular}{c}
  Global\\
  Asymptotic\\
  Stability
\end{tabular}
$\implies$
\begin{tabular}{c}
  Uniform\\
  Stability
\end{tabular}
\end{center}

The structure of the paper is as follows.
In Section~\ref{dal}, we start with fundamental properties related to
the fundamental solution of \eqref{introeq1},
which will be required in the sequel.
Section~\ref{srua1} and Section~\ref{srua2}
include explicit conditions for uniform stability,
global asymptotic stability and
uniform exponential stability
under the primary assumptions \ref{a1} and \ref{a2}, respectively.
In Section~\ref{sapp}, we provide three examples to show applicability
of the new results on some nonstandard time scales.
Section~\ref{findis} includes some directions for
future research.
Section~\ref{appx} is the appendix.
We describe auxiliary results
related to the fundamental solution in Subsection~\ref{appa},
the time scales exponential function in Subsection~\ref{appb} and
the basic time scales calculus is presented in Subsection~\ref{appc}.

\section{Definitions and Auxiliary Results}\label{dal}

In this section, we first introduce the notion of the fundamental solution and
the variation of parameters formula for \eqref{introeq1}.
Then, we will give three main theorems on the stability of \eqref{introeq1}.

\begin{definition}[Fundamental Solution]\label{daldf1}
For $s\in[t_{0},\infty)_{\T}$, the solution
$\fun=\fun(\cdot,s):[\dly_{\ast}(s),\infty)_{\T}\to\R$ of the initial value problem
\begin{equation}
\begin{cases}
x^{\Delta}(t)+\cf(t)x(\dly(t))=0\quad\text{for}\ t\in[s,\infty)_{\T},\\
x(s)=1\quad\text{and}\quad{}x(t)\equiv0\quad\text{for}\ t\in[\dly_{\ast}(s),s)_{\T}
\end{cases}\notag
\end{equation}
is called the \textbf{fundamental solution} of \eqref{introeq1}.
\end{definition}

The following result can be found in \cite[Lemma~2.2]{MR2683912}.

\begin{lemma}[Solution Representation]\label{dallm1}
Let $s\in[t_{0},\infty)_{\T}$ and $x$ be the solution of the initial value problem
\begin{equation}
\begin{cases}
x^{\Delta}(t)+\cf(t)x(\dly(t))=\ftrm(t) \quad\text{for}\ t\in[s,\infty)_{\T},\\
x(s)=\initval\quad\text{and}\quad{}x(t)=\initfun(t)\quad\text{for}\ t\in[\dly_{\ast}(s),s)_{\T},
\end{cases}\notag
\end{equation}
then
\begin{equation}
x(t)=\fun(t,s)\initval-\int_{s}^{t}\fun(t,\sigma(\eta))\cf(\eta)\initfun(\dly(\eta))\Delta\eta
+\int_{s}^{t}\fun(t,\sigma(\eta))\ftrm(\eta)\Delta\eta
\notag
\end{equation}
for $t\in[s,\infty)_{\T}$, where $\fun$ is the fundamental solution of \eqref{introeq1}.
We assume above that functions vanish out of their specified domains, i.e., $\initfun(t)\equiv0$ for $t\in[s,\infty)_{\T}$.
\end{lemma}

For the next theorem, we introduce the condition
\begin{equation}
\sup_{s\in[t_{0},\infty)_{\T}}\negmedspace\bigg\{\negmedspace\int_{s}^{\dly_{-1}(s)}\cf(\eta)\Delta\eta\bigg\}<\infty, \label{daleq1}
\end{equation}
which is equivalent to $\limsup_{s\to\infty}\int_{s}^{\dly_{-1}(s)}\cf(\eta)\Delta\eta<\infty$.


\begin{theorem}\label{dalthm1}
Assume that \eqref{daleq1} holds.
Then, the following statements are equivalent.
\begin{enumerate}[label={(\roman*)},leftmargin={*},ref=(\roman*)]
\item\label{dalthm1it1} The trivial solution of \eqref{introeq1} is uniformly stable.
\item\label{dalthm1it2} There exists $M_{0}\in\R^{+}$ such that
    \begin{equation}
    |\fun(t,s)|\leq{}M_{0}\quad\text{for all}\ (t,s)\in\Lambda_{t_{0}},\label{dalthm1it2eq1}
    \end{equation}
    where
    \begin{equation}
    \Lambda_{t_{0}}:=\{(t,s)\in\T\times\T:\ t\geq{}s\geq{}t_{0}\}.\label{dalthm1it2eq2}
    \end{equation}
\end{enumerate}
\end{theorem}

\begin{proof}
\noindent\ref{dalthm1it1}$\implies$\ref{dalthm1it2}
Let the trivial solution of \eqref{introeq1} be uniformly stable.
Given $\varepsilon\in\R^{+}$, there exists $\delta\in\R^{+}$ such that for any $s\in[t_{0},\infty)_{\T}$,
any solution $x$ of \eqref{predf2eq1} with $|\initval|+\sup_{\eta\in{}[\dly_{\ast}(s),s)_{\T}}|\initfun(\eta)|<\delta$
satisfies $|x(t)|<\varepsilon$ for all $t\in[s,\infty)_{\T}$.
For a fixed $s\in[t_{0},\infty)_{\T}$ and the solution $x(t):=\frac{\delta}{2}\fun(t,s)$ for $t\in[s,\infty)_{\T}$,
we see that $|x(t)|<\varepsilon$
for all $t\in[s,\infty)_{\T}$ since $|x(s)|+\sup_{\eta\in{}[\dly_{\ast}(s),s)_{\T}}|x(\eta)|=\frac{\delta}{2}<\delta$.
Hence, \eqref{dalthm1it2eq1} holds with $M_{0}:=\frac{2\varepsilon}{\delta}$.
\smallskip

\noindent\ref{dalthm1it2}$\implies$\ref{dalthm1it1}
Let $K_{0}:=\sup_{s\in[t_{0},\infty)_{\T}}\negmedspace\big\{\negmedspace\int_{s}^{\dly_{-1}(s)}\cf(\eta)\Delta\eta\big\}$.
Using Lemma~\ref{dallm1} and the vanishing property of the initial function $\initfun$, we have for all $t\in[s,\infty)_{\T}$
\begin{align}
|x(t)|\leq&M_{0}\bigg(\initval+\int_{s}^{t}\cf(\eta)\initfun(\dly(\eta))\Delta\eta\bigg)
\leq{}M_{0}\bigg(\initval+\int_{s}^{\dly_{-1}(s)}\cf(\eta)\initfun(\dly(\eta))\Delta\eta\bigg)\notag\\
\leq&M_{0}\bigg(\initval+K_{0}\sup_{\eta\in{}[\dly_{\ast}(s),s)_{\T}}|\initfun(\eta)|\bigg)
\leq{}M_{0}(K_{0}+1)\bigg(\initval+\sup_{\eta\in{}[\dly_{\ast}(s),s)_{\T}}|\initfun(\eta)|\bigg),\notag
\end{align}
from which the uniform stability of the trivial solution of \eqref{introeq1} follows.
\end{proof}

\begin{theorem}\label{dalthm2}
Assume that \eqref{daleq1} holds.
Then, the following statements are equivalent.
\begin{enumerate}[label={(\roman*)},leftmargin={*},ref=(\roman*)]
\item\label{dalthm2it1} The trivial solution of \eqref{introeq1} is globally attracting.
\item\label{dalthm2it2} $\lim_{t\to\infty}\fun(t,s)=0$ for any $s\in[t_{0},\infty)_{\T}$.
\end{enumerate}
\end{theorem}

\begin{proof}

\noindent\ref{dalthm2it1}$\implies$\ref{dalthm2it2}
If the trivial solution of \eqref{introeq1} is globally attracting,
then it is obvious from Definition~\ref{dfnga} that $\lim_{t\to\infty}\fun(t,s)=0$ for any $s\in[t_{0},\infty)_{\T}$.
\smallskip

\noindent\ref{dalthm2it2}$\implies$\ref{dalthm2it1}
By Theorem~\ref{althm1}, the fundamental solution $\fun$ (as a function of two variables)
is continuous in (the triangular domain) $\Lambda_{t_{0}}$, thus we have
\begin{equation}
\lim_{t\to\infty}\max_{\eta\in{}[s,\dly_{-1}(s)]_{\T}}|\fun(t,\eta)|=0
\quad\text{for any}\ s\in[t_{0},\infty)_{\T}.\notag
\end{equation}
The claim follows immediately from the inequality
\begin{align}
|x(t)|\leq&|\fun(t,s)||\initval|+\int_{s}^{\dly_{-1}(s)}|\fun(t,\sigma(\eta))|\cf(\eta)|\initfun(\dly(\eta))|\Delta\eta\notag\\
\leq&|\fun(t,s)||\initval|+K_{0}\bigg(\max_{\eta\in{}[s,\dly_{-1}(s)]_{\T}}|\fun(t,\eta)|\bigg)\bigg(\sup_{\eta\in{}[\dly_{\ast}(s),s)_{\T}}|\initfun(\eta)|\bigg)\notag
\end{align}
for all $t\in[\dly_{-1}(s),\infty)_{\T}$,
where $K_{0}:=\sup_{s\in[t_{0},\infty)_{\T}}\negmedspace\big\{\negmedspace\int_{s}^{\dly_{-1}(s)}\cf(\eta)\Delta\eta\big\}$.
\end{proof}

Now, we require the condition
\begin{equation}
\sup_{s\in[t_{0},\infty)_{\T}}\{\dly_{-1}(s)-s\}<\infty,\label{daleq2}
\end{equation}
which is equivalent to $\limsup_{s\to\infty}[\dly_{-1}(s)-s]<\infty$.

\begin{theorem}\label{dalthm3}
Assume that \eqref{daleq1} and \eqref{daleq2} hold.
Then, the following statements are equivalent.
\begin{enumerate}[label={(\roman*)},leftmargin={*},ref=(\roman*)]
\item\label{dalthm3it1} The trivial solution of \eqref{introeq1} is uniformly exponentially stable.
\item\label{dalthm3it2} There exist $M_{0},\lambda_{0}\in\R^{+}$ such that
    \begin{equation}
    |\fun(t,s)|\leq{}M_{0}\mathrm{e}_{\ominus\lambda_{0}}(t,s)\quad\text{for all}\ (t,s)\in\Lambda_{t_{0}},\label{dalthm3it2eq1}
    \end{equation}
    where $\Lambda_{t_{0}}$ is defined in \eqref{dalthm1it2eq2}.
\end{enumerate}
\end{theorem}

\begin{proof}
\noindent\ref{dalthm3it1}$\implies$\ref{dalthm3it2}
If the trivial solution of \eqref{introeq1} is uniformly stable,
then it is obvious from Definitions~\ref{dfnues} and \ref{daldf1} that \eqref{dalthm3it2eq1} holds.
\smallskip

\noindent\ref{dalthm3it2}$\implies$\ref{dalthm3it1}
Let
\begin{equation}
K_{0}:=\sup_{s\in[t_{0},\infty)_{\T}}\negmedspace\bigg\{\negmedspace\int_{s}^{\dly_{-1}(s)}\cf(\eta)\Delta\eta\bigg\}
\quad\text{and}\quad
H_{0}:=\sup_{s\in[t_{0},\infty)_{\T}}\{\dly_{-1}(s)-s\}.\notag
\end{equation}
Hence, Lemma~\ref{dallm1} and the vanishing property of the initial function $\initfun$ imply for all $t\in[s,\infty)_{\T}$ that
\begin{align*}
|x(t)|\leq&M_{0}\mathrm{e}_{\ominus\lambda_{0}}(t,s)\bigg(\initval+\int_{s}^{t}\mathrm{e}_{\ominus\lambda_{0}}(s,\sigma(\eta))\cf(\eta)\initfun(\dly(\eta))\Delta\eta\bigg)
\\
=&M_{0}\mathrm{e}_{\ominus\lambda_{0}}(t,s)\bigg(\initval+\int_{s}^{t}\mathrm{e}_{\lambda_{0}}(\sigma(\eta),s)\cf(\eta)\initfun(\dly(\eta))\Delta\eta\bigg)
\\
\leq&M_{0}\mathrm{e}_{\ominus\lambda_{0}}(t,s)\bigg(\initval+\int_{s}^{\dly_{-1}(s)}\mathrm{e}_{\lambda_{0}}(\sigma(\eta),s)\cf(\eta)\initfun(\dly(\eta))\Delta\eta\bigg)\\
\leq&M_{0}\mathrm{e}_{\ominus\lambda_{0}}(t,s)\bigg(\initval+\mathrm{e}^{\lambda_{0}H_{0}}\int_{s}^{\dly_{-1}(s)}\cf(\eta)\initfun(\dly(\eta))\Delta\eta\bigg)\\
\leq&M_{0}\mathrm{e}_{\ominus\lambda_{0}}(t,s)\bigg(\initval+K_{0}\mathrm{e}^{\lambda_{0}H_{0}}\sup_{\eta\in{}[\dly_{\ast}(s),s)_{\T}}|\initfun(\eta)|\bigg)\\
\leq&M_{0}\mathrm{e}_{\ominus\lambda_{0}}(t,s)\big(K_{0}\mathrm{e}^{\lambda_{0}H_{0}}+1\big)\bigg(\initval+\sup_{\eta\in{}[\dly_{\ast}(s),s)_{\T}}|\initfun(\eta)|\bigg),
\end{align*}
from which the uniform exponential stability of the trivial solution of \eqref{introeq1} follows.
Note that we have used \cite[Lemma~3.2]{MR2842561} in the third step.
It should be noted here that the integral variable satisfies $\eta\in[s,\dly_{-1}(s))_{\T}$,
which implies $\sigma(\eta)\in[s,\dly_{-1}(s)]_{\T}$.
\end{proof}

\begin{remark}
Note that the implication \ref{dalthm1it1}$\Rightarrow$\ref{dalthm1it2} of Theorem~\ref{dalthm1} and \ref{dalthm2it1}$\Rightarrow$\ref{dalthm2it2} of Theorem~\ref{dalthm2} hold without the additional
assumption \eqref{daleq1},
and the part \ref{dalthm3it1}$\Rightarrow$\ref{dalthm3it2} of Theorem~\ref{dalthm3} requires neither \eqref{daleq1} nor \eqref{daleq2}.
\end{remark}

Clearly, Theorem~\ref{dalthm3} improves \cite[Theorem~4.1]{MR2927064}.

\begin{remark}\label{dalrmk2}
Consider the conditions
\begin{equation}
\sup_{s\in[t_{0},\infty)_{\T}}\negmedspace\bigg\{\negmedspace\int_{s}^{\infty}\cf(\eta)\chi_{(-\infty,s)_{\T}}(\dly(\eta))\Delta\eta\bigg\}<\infty\label{dalrmk2eq1}
\end{equation}
and for all fixed $\lambda\in\R^{+}$
\begin{equation}
\sup_{s\in[t_{0},\infty)_{\T}}\negmedspace\bigg\{\negmedspace\int_{s}^{\infty}\mathrm{e}_{\lambda}(\sigma(\eta),s)
\cf(\eta)\chi_{(-\infty,s)_{\T}}(\dly(\eta))\Delta\eta\bigg\}<\infty,
\label{dalrmk2eq2}
\end{equation}
where $\chi_{D}:D\to\{0,1\}$ is the characteristic function of the set $D\subset\R$,
i.e., $\chi_{D}(t)=1$ for $t\in{}D$ and $\chi_{D}(t)=0$ for $t\not\in{}D$.
The condition $\lim_{t\to\infty}\dly(t)=\infty$ and the function $\dly_{-1}$ in Theorems~\ref{dalthm1}, \ref{dalthm2} and \ref{dalthm3} can be omitted
by assuming \eqref{dalrmk2eq1} and \eqref{dalrmk2eq2} instead of \eqref{daleq1} and \eqref{daleq2}, respectively.
\end{remark}

The following example demonstrates that the conditions \eqref{dalrmk2eq1} and \eqref{dalrmk2eq2}
(thus \eqref{daleq1} and \eqref{daleq2}) are crucial in Theorems~\ref{dalthm1}, \ref{dalthm2} and \ref{dalthm3},
as well as the condition $\lim_{t\to\infty}\dly(t)=\infty$.

\begin{example}\label{dalex1}
Consider the time scale $\mathbb{P}_{1,1}=\cup_{k\in\Z}[2k,2k+1]_{\R}=\cdots\cup[0,1]_{\R}\cup[2,3]_{\R}\cup\cdots$ and the dynamic equation
\begin{equation}
x^{\Delta}(t)+x(\dly(t))=0\quad\text{for}\ t\in[0,\infty)_{\mathbb{P}_{1,1}},\label{dalex1eq1}
\end{equation}
where $\dly(t):=t$ if
$\mu(t)=0$ and $\dly(t):=-1$ if $\mu(t)=1$ for $t\in[0,\infty)_{\mathbb{P}_{1,1}}$,
where $\mu$ is the graininess function defined in Section~\ref{appc}.
We show below that \eqref{dalrmk2eq1} does not hold.
Simply, we have
\begin{equation}
\chi_{(-\infty,s)_{\T}}(\dly(t))
=
\begin{cases}
\chi_{(-\infty,s)_{\T}}(t)=0,&t\geq{}s,\ t\in[2k,2k+1)_{\R}\ \text{and}\ k\in\N_{0},\\
\chi_{(-\infty,s)_{\T}}(-1)=1,&t\geq{}s,\ t=2k+1\ \text{and}\ k\in\N_{0}
\end{cases}\notag
\end{equation}
for $s\in[t_{0},\infty)_{\T}$.
By \cite[Theorem~1.75]{MR1843232}, we get
\begin{equation}
\int_{s}^{\infty}\chi_{(-\infty,s)_{\T}}(\dly(\eta))\Delta\eta=\sum_{\eta\in[s,\infty)_{2\N_{0}+1}}\mu(\eta)=\infty\quad\text{for}\ s\in[t_{0},\infty)_{\T}.\notag
\end{equation}
On the other hand, examining \eqref{dalex1eq1} yields the system
\begin{equation}
\begin{cases}
x^{\prime}(t)+x(t)=0\quad\text{for}\ t\in[2k,2k+1)_{\R}\ \text{and}\ k\in\N_{0},\\
\Delta{}x(2k+1)+x(-1)=0\quad\text{for}\ k\in\N_{0},
\end{cases}\notag
\end{equation}
whose solution is
\begin{equation}
x(t)=
\begin{cases}
\frac{1}{\mathrm{e}^{k}}x(0)-\frac{\mathrm{e}}{\mathrm{e}-1}\big(1-\frac{1}{\mathrm{e}^{k}}\big)x(-1),&t=2k\ \text{and}\
k\in\N_{0},\\
x(2k)\mathrm{e}^{-(t-2k)},&t\in(2k,2k+1]_{\R}\ \text{and}\ k\in\N_{0}.
\end{cases}\notag
\end{equation}
Clearly, $\limsup_{t\to\infty}|x(t)|=\frac{\mathrm{e}}{\mathrm{e}-1}|x(-1)|>0$ provided that $x(-1)\neq0$.
We can easily show that the fundamental solution $\fun$ of \eqref{dalex1eq1} is positive.
Further, we can estimate that $0\leq\fun(t,s)\leq\mathrm{e}{\thinspace}\mathrm{e}^{-(\frac{t-s}{2})}$ for all $(t,s)\in\Lambda_{0}$,
where $\Lambda$ is defined in \eqref{dalthm1it2eq2}.
This implies $\lim_{t\to\infty}\fun(t,s)=0$ uniformly in $s\in[0,\infty)_{\mathbb{P}_{1,1}}$
but the trivial solution of \eqref{dalex1eq1} is not globally attracting.
Thus, the conclusion of Theorem~\ref{dalthm2} may not be valid without \eqref{dalrmk2eq1}.

Further, \eqref{dalrmk2eq2} is not fulfilled for this example.
Since the graininess function is bounded ($\mu(t)\leq1$ for all $t\in\mathbb{P}_{1,1}$),
by \cite[Lemma~2.3]{MR2927064},
we see that the fundamental solution of \eqref{dalex1eq1}
satisfies the exponential estimate \eqref{dalthm3it2eq1} but not \eqref{dalrmk2eq2}.
The trivial solution of \eqref{introeq1} is not uniformly exponentially stable.
Thus, the conclusion of Theorem~\ref{dalthm3} may also not be valid without \eqref{daleq1} and \eqref{daleq2}.
\end{example}

%

\section{Stability Results under {\protect\ref{a1}}}\label{srua1}

In this section, we will provide stability results under the condition \ref{a1}.
We will start with a technical lemma and then estimate the fundamental solution.
And, finally, the last three subsections of this section will provide explicit conditions for uniform stability,
global asymptotic stability and uniform exponential stability, respectively.

\subsection{A Technical Lemma}\label{stla1}


\begin{lemma}\label{stla1lm1}
Assume \ref{a1} and
\begin{equation}
\sup_{t\in[t_{0},\infty)_{\T}}\negmedspace\bigg\{\negmedspace\int_{\dly_{\ast}(t)}^{\sigma(t)}\cf(\eta)\Delta\eta\bigg\}<1.\label{stla1lm1eq1}
\end{equation}
Then, there exists $\lambda_{0}\in(0,1)_{\R}$ such that
\begin{equation}
\int_{\dly_{\ast}(t)}^{\sigma(t)}\mathrm{e}_{\lambda_{0}\cf}(t,\dly_{\ast}(\eta))\cf(\eta)\Delta\eta
<\frac{1-\lambda_{0}}{1+\lambda_{0}\cf(t)\mu(t)}\quad\text{for all}\ t\in[t_{0},\infty)_{\T}.\notag
\end{equation}
\end{lemma}

\begin{proof}
It follows from \eqref{stla1lm1eq1} that there exists $\nu_{0}\in(0,1)_{\R}$ such that
\begin{equation}
\int_{\dly_{\ast}(t)}^{\sigma(t)}\cf(\eta)\Delta\eta<\nu_{0}\quad\text{for all}\ t\in[t_{0},\infty)_{\T},\label{mrlmprfeq1}
\end{equation}
which implies
\begin{equation}
\cf(t)\mu(t)<\nu_{0}\quad\text{and}\quad\int_{\dly_{\ast}^{2}(t)}^{\sigma(t)}\cf(\eta)\Delta\eta<2\nu_{0}
\quad\text{for all}\ t\in[t_{0},\infty)_{\T}.\label{mrlmprfeq2}
\end{equation}
Now, we can estimate
\begin{equation}
\int_{\dly_{\ast}(t)}^{\sigma(t)}\mathrm{e}_{\lambda\cf}(t,\dly_{\ast}(\eta))\cf(\eta)\Delta\eta
\leq\mathrm{e}_{\lambda\cf}(t,\dly_{\ast}^{2}(t))\int_{\dly_{\ast}(t)}^{\sigma(t)}\cf(\eta)\Delta\eta\notag
\end{equation}
for all $\lambda\in(0,1)_{\R}$ and all $t\in[t_{0},\infty)_{\T}$.
This yields by using \cite[Lemma~2.3]{MR2842561}, \eqref{stla1lm1eq1} and \eqref{mrlmprfeq2} that
\begin{equation}
\int_{\dly_{\ast}(t)}^{\sigma(t)}\mathrm{e}_{\lambda\cf}(t,\dly_{\ast}(\eta))\cf(\eta)\Delta\eta<\nu_{0}\mathrm{e}^{2\lambda\nu_{0}}\label{mrlmprfeq5}
\end{equation}
for all $\lambda\in(0,1)_{\R}$ and all $t\in[t_{0},\infty)_{\T}$.

Let us define the function $\phi\in\cnt{}([0,1]_{\R},\R)$ by the formula
\begin{equation}
\phi(\lambda)=\frac{1-\lambda}{1+\lambda\nu_{0}}-\nu_{0}\mathrm{e}^{2\lambda\nu_{0}}\quad\text{for}\ \lambda\in[0,1]_{\R}.\notag
\end{equation}
Clearly, $\phi(0)=1-\nu_{0}>0$ and $\phi(1)=-\nu_{0}\mathrm{e}^{2\nu_{0}}<0$.
Therefore, we may find $\lambda_{0}\in(0,1)_{\R}$ such that $\phi(\lambda_{0})=0$.
Using \eqref{mrlmprfeq2} and \eqref{mrlmprfeq5}, we have
\begin{align}
\int_{\dly_{\ast}(t)}^{\sigma(t)}\mathrm{e}_{\lambda_{0}\cf}(t,\dly_{\ast}(\eta))\cf(\eta)\Delta\eta
<&\nu_{0}\mathrm{e}^{2\lambda_{0}\nu_{0}}
=\frac{1-\lambda_{0}}{1+\lambda_{0}\nu_{0}}-\phi(\lambda_{0})
=\frac{1-\lambda_{0}}{1+\lambda_{0}\nu_{0}}\notag\\
\leq&\frac{1-\lambda_{0}}{1+\lambda_{0}\cf(t)\mu(t)}\notag
\end{align}
for all $t\in[t_{0},\infty)_{\T}$,
which concludes the proof.
\end{proof}

\subsection{Some Properties of the Fundamental Solution}\label{spfsa1}

\begin{lemma}\label{spfsa1lm1}
Assume \ref{a1} and
\begin{equation}
\int_{\dly(t)}^{\sigma(t)}\cf(\eta)\Delta\eta
\leq1\quad\text{for all}\ t\in[t_{0},\infty)_{\T}.\label{spfsa1lm1eq1}
\end{equation}
Then,
\begin{equation}
|\fun(t,s)|\leq1\quad\text{for all}\ (t,s)\in\Lambda_{t_{0}},\label{spfsa1lm1eq2}
\end{equation}
where $\Lambda_{t_{0}}$ is defined in \eqref{dalthm1it2eq2}.
\end{lemma}

\begin{proof}
Fix $s\in[t_{0},\infty)_{\T}$ and denote $x(t):=\fun(t,s)$ for $t\in[s,\infty)_{\T}$.
First, let $x$ be positive on $[s,\infty)_{\T}$,
then $x$ is nonincreasing on $[s,\infty)_{\T}$,
which implies $0<x(t)\leq1$ for all $t\in[s,\infty)_{\T}$.
Thus for a positive $x$ the claim is true.
Next, let $x$ have some generalized zeros on $[s,\infty)_{\T}$, i.e.,
there exists $t_{1}\in[s,\infty)_{\T}$ such that either $x(t_{1})=0$ or $x(t_{1})>0$ and $x^{\sigma}(t_{1})<0$.
Hence, $0\leq{}x(t)\leq1$ for all $t\in[s,t_{1}]_{\T}$.
Let
\begin{equation}
t_{2}:=\sup\{t\in[s,\infty)_{\T}:\ |x(\eta)|\leq1\ \text{for all}\ \eta\in[s,t]_{\T}\}.\notag
\end{equation}
Clearly, $t_{2}\geq{}t_{1}$.
To prove $t_{2}=\infty$, assume the contrary that $t_{2}$ is finite.
Assume for now that $t_{2}$ is right-scattered, i.e., $\mu(t_{2})>0$.
Then, we have $|x^{\sigma}(t_{2})|>1$ and $|x(t)|\leq1$ for all $t\in[s,t_{2}]_{\T}$.
Without loss of generality, let $x^{\sigma}(t_{2})>1$.
The case where $x^{\sigma}(t_{2})<1$ is treated similarly.
Thus, $x^{\Delta}(t_{2})=\frac{x^{\sigma}(t_{2})-x(t_{2})}{\mu(t_{2})}>0$.
From \eqref{introeq1}, we have $x(\dly(t_{2}))<0$.
Integrating \eqref{introeq1} from $\alpha(t_{2})$ to $\sigma(t_{2})$, we get
\begin{equation}
x^{\sigma}(t_{2})
=x(\dly(t_{2}))-\int_{\dly(t_{2})}^{\sigma(t_{2})}\cf(\eta)x(\dly(\eta))\Delta\eta
<\int_{\dly(t_{2})}^{\sigma(t_{2})}\cf(\eta)\Delta\eta\leq1,\notag
\end{equation}
which contradicts $x^{\sigma}(t_{2})>1$
(note that $\eta\in[\dly(t_{2}),\sigma(t_{2}))_{\T}$ implies $\dly(\eta)\leq t_{2}$).
This shows that $t_{2}$ is right-dense.
That is, $x$ is continuous at $t_{2}$.
In this case, $|x(t_{2})|=1$.
Hence, we can find $t_{3}\in(t_{2},\infty)_{\T}$
such that $|x(t_{3})|>1$ and $x$ is of fixed sign on $[t_{2},t_{3}]_{\T}$.
Without loss of generality, assume that $x(t_{2})=1$ and $x(t)>0$ for all $t\in[t_{2},t_{3}]_{\T}$
(the case where $x(t_{2})=-1$ and $x(t)<0$ for all $t\in[t_{2},t_{3}]_{\T}$ is treated similarly).
Let $t_{4}$ be the greatest generalized zero of $x$ on $[s,t_{2})_{\T}$.
Hence, we have either $x(t_{4})=0$ or $x(t_{4})<0$ and $x^{\sigma}(t_{4})>0$.
Further, $x(t)>0$ for all $t\in(t_{4},t_{3}]_{\T}$.
We can also find $t_{5}\in[t_{2},t_{3}]_{\T}$ such that $x^{\Delta}(t_{5})>0$ and $x^{\sigma}(t_{5})>1$.
This implies $x(\dly(t_{5}))<0$ by \eqref{introeq1}.
If $t_{4}<\dly(t_{5})\leq{}t_{5}$,
then $x(\dly(t_{5}))>0$,
which is contradiction.
Thus, $\dly(t_{5})\leq{}t_{4}$.
Further, we have $|x(\dly(t))|\leq1$ for all $t\in[t_{4},t_{5}]_{\T}$.
So, integrating \eqref{introeq1} from $t_{4}$ to $\sigma(t_{5})$ yields
\begin{equation}
x^{\sigma}(t_{5})=x(t_{4})-\int_{t_{4}}^{\sigma(t_{5})}\cf(\eta)x(\dly(\eta))\Delta\eta
\leq\int_{t_{4}}^{\sigma(t_{5})}\cf(\eta)\Delta\eta
\leq\int_{\dly(t_{5})}^{\sigma(t_{5})}\cf(\eta)\Delta\eta\leq1,\notag
\end{equation}
which contradicts $x^{\sigma}(t_{5})>1$.
This implies $t_{2}=\infty$ and completes the proof.
\end{proof}

\begin{lemma}\label{spfsa1lm2}
Assume \ref{a1}, \eqref{daleq1} and \eqref{stla1lm1eq1}.
Then,
\begin{equation}
|\fun(t,s)|\leq{}M_{0}\mathrm{e}_{\ominus(\lambda_{0}\cf)}(t,s)\quad\text{for all}\ (t,s)\in\Lambda_{t_{0}},\notag
\end{equation}
where $\Lambda_{t_{0}}$ is defined in \eqref{dalthm1it2eq2}, $M_{0}\in\R^{+}$ and $\lambda_{0}\in(0,1)_{\R}$ is the number provided by Lemma~\ref{stla1lm1}.
\end{lemma}

\begin{proof}
For simplicity of notation, fix $s\in[t_{0},\infty)_{\T}$ and let $x(t):=\fun(t,s)$ for $t\in[s,\infty)_{\T}$.
From \eqref{stla1lm1eq1}, the claim of Lemma~\ref{stla1lm1} holds with some $\lambda_{0}\in(0,1)_{\R}$ on $[s,\infty)_{\T}$.
Using \eqref{introeq1}, we have
\begin{align}
x^{\Delta}(t)=&-\cf(t)x^{\sigma}(t)
+\cf(t)\int_{\dly(t)}^{\sigma(t)}x^{\Delta}(\eta)\Delta\eta\notag\\
=&-\cf(t)x^{\sigma}(t)
-\cf(t)\int_{\dly(t)}^{\sigma(t)}\cf(\eta)x(\dly(\eta))\Delta\eta\label{spfsa1lm2prfeq1}
\end{align}
for all $t\in[\dly_{-1}(s),\infty)_{\T}$.
Applying the solution representation formula in Lemma~\ref{dallm1}
for \eqref{spfsa1lm2prfeq1},
we get
\begin{equation}
\begin{aligned}[]
x(t)=&x(\dly_{-1}(s))\mathrm{e}_{\ominus\cf}(t,\dly_{-1}(s))
-\int_{\dly_{-1}(s)}^{t}\mathrm{e}_{\ominus\cf}(t,\eta)\cf(\eta)\\
&\phantom{x(\dly_{-1}(s))\mathrm{e}_{\ominus\cf}(t,\dly_{-1}(s))-\int_{\dly_{-1}(s)}^{t}}
\times\int_{\dly(\eta)}^{\sigma(\eta)}\cf(\zeta)x(\dly(\zeta))\Delta\zeta\Delta\eta
\end{aligned}\label{spfsa1lm2prfeq2}
\end{equation}
for all $t\in[\dly_{-1}(s),\infty)_{\T}$.
Let us define a function $y\in\crd{1}([s,\infty)_{\T},\R)$ by
\begin{equation}
y(t):=
\begin{cases}
x(t)\mathrm{e}_{\lambda_{0}\cf}(t,\dly_{-1}(s)),&t\in[\dly_{-1}(s),\infty)_{\T},\\
x(t),&t\in[s,\dly_{-1}(s)]_{\T}.
\end{cases}\label{spfsa1lm2prfeq3}
\end{equation}
Multiplying \eqref{spfsa1lm2prfeq2} by $\mathrm{e}_{\lambda_{0}\cf}(\cdot,\dly_{-1}(s))$, we get
\begin{equation}
\begin{aligned}[]
y(t)=&y(\dly_{-1}(s))\mathrm{e}_{(\lambda_{0}\cf)\ominus\cf}(t,\dly_{-1}(s))\\
&-\int_{\dly_{-1}(s)}^{t}\mathrm{e}_{(\lambda_{0}\cf)\ominus\cf}(t,\eta)\cf(\eta)
\int_{\dly(\eta)}^{\sigma(\eta)}\mathrm{e}_{\lambda_{0}\cf}(\eta,\dly(\zeta))\cf(\zeta)y(\dly(\zeta))\Delta\zeta\Delta\eta
\end{aligned}\label{spfsa1lm2prfeq4}
\end{equation}
for all $t\in[\dly_{-1}(s),\infty)_{\T}$.

Next, we claim that $|y(t)|\leq1$ for all $t\in[s,\infty)_{\T}$.
Let $$t_{2}:=\sup\{t\in[s,\infty)_{\T}:\ |y(\eta)|\leq1\ \text{for all}\ \eta\in[s,t]_{\T}\}.$$
To prove $t_{2}=\infty$,
assume the contrary that $t_{2}$ is finite.
Clearly, $t_{2}>\dly_{-1}(s)$ by \eqref{spfsa1lm2prfeq3} and Lemma~\ref{spfsa1lm1}.
Assume for now that $t_{2}$ is right-scattered.
Then, we have $|y^{\sigma}(t_{2})|>1$ and $|y(t)|\leq1$ for all $t\in[s,t_{2}]_{\T}$.
Using \eqref{spfsa1lm2prfeq4} and Lemma~\ref{stla1lm1}, we obtain
\begin{align}
\begin{split}
|y^{\sigma}(t_{2})|\leq&\mathrm{e}_{(\lambda_{0}\cf)\ominus\cf}(\sigma(t_{2}),\dly_{-1}(s))
+\int_{\dly_{-1}(s)}^{\sigma(t_{2})}\mathrm{e}_{(\lambda_{0}\cf)\ominus\cf}(\sigma(t_{2}),\eta)\cf(\eta)\\
&\phantom{\mathrm{e}_{(\lambda_{0}\cf)\ominus\cf}(\sigma(t_{2}),\dly_{-1}(s))+\int_{\dly_{-1}(s)}^{\sigma(t_{2})}}
\times\int_{\dly_{\ast}(\eta)}^{\sigma(\eta)}\mathrm{e}_{\lambda_{0}\cf}(\eta,\dly_{\ast}(\zeta))\cf(\zeta)\Delta\zeta\Delta\eta
\end{split}\notag\\
<&\mathrm{e}_{(\lambda_{0}\cf)\ominus\cf}(\sigma(t_{2}),\dly_{-1}(s))
+\int_{\dly_{-1}(s)}^{\sigma(t_{2})}\mathrm{e}_{(\lambda_{0}\cf)\ominus\cf}(\sigma(t_{2}),\eta)\frac{(1-\lambda_{0})\cf(\eta)}{1+\lambda_{0}\cf(\eta)\mu(\eta)}\Delta\eta\notag\\
=&\mathrm{e}_{(\lambda_{0}\cf)\ominus\cf}(\sigma(t_{2}),\dly_{-1}(s))
+\int_{\dly_{-1}(s)}^{\sigma(t_{2})}\mathrm{e}_{(\lambda_{0}\cf)\ominus\cf}(\sigma(t_{2}),\eta)\big(\cf\ominus{}(\lambda_{0}\cf)\big)(\eta)\Delta\eta\notag\\
=&\mathrm{e}_{(\lambda_{0}\cf)\ominus\cf}(\sigma(t_{2}),\dly_{-1}(s))
+\int_{\dly_{-1}(s)}^{\sigma(t_{2})}\frac{\Delta}{\Delta\eta}\mathrm{e}_{(\lambda_{0}\cf)\ominus\cf}(\sigma(t_{2}),\eta)\Delta\eta=1,\notag
\end{align}
which is a contradiction.
This shows that $t_{2}$ is right-dense.
That is, $y$ is continuous at $t_{2}$.
In this case, $|y(t_{2})|=1$.
Using Lemma~\ref{stla1lm1} and \eqref{spfsa1lm2prfeq4},
we can proceed as above and show that $|y(t_{2})|<1$,
which is also a contradiction.
Thus, $t_{2}=\infty$, i.e.,
$|y(t)|\leq1$ for all $t\in[s,\infty)_{\T}$.

It follows from \eqref{spfsa1lm2prfeq3} that
\begin{align}
|x(t)|\leq&\mathrm{e}_{\ominus(\lambda_{0}\cf)}(t,\dly_{-1}(s))|y(t)|
\leq\mathrm{e}_{\ominus(\lambda_{0}\cf)}(t,\dly_{-1}(s))\notag\\
=&\mathrm{e}_{\lambda_{0}\cf}(\dly_{-1}(s),s)\mathrm{e}_{\ominus(\lambda_{0}\cf)}(t,s)\label{spfsa1lm2prfeq6}
\end{align}
for all $t\in[\dly_{-1}(s),\infty)_{\T}$.
By \cite[Lemma~3.2]{MR2842561} and \eqref{daleq1}, we estimate
\begin{equation}
\mathrm{e}_{\lambda_{0}\cf}(\dly_{-1}(s),s)
\leq\exp\bigg\{\lambda_{0}\int_{s}^{\dly_{-1}(s)}\cf(\eta)\Delta\eta\bigg\}\leq\mathrm{e}^{\lambda_{0}K_{0}}=:M_{0},
\notag
\end{equation}
where $K_{0}:=\sup_{s\in[t_{0},\infty)_{\T}}\negmedspace\big\{\negmedspace\int_{s}^{\dly_{-1}(s)}\cf(\eta)\Delta\eta\big\}$.
Hence, by \eqref{spfsa1lm2prfeq6}, $M_{0}>1$ and the fact that $|x(t)|\leq1$ for all $t\in[s,\dly_{-1}(s)]_{\T}$,
we have $|x(t)|\leq{}M_{0}\mathrm{e}_{\ominus(\lambda_{0}\cf)}(t,s)$ for all $t\in[s,\infty)_{\T}$,
which completes the proof.
\end{proof}

\subsection{Uniform Stability}\label{usa1}

\begin{theorem}\label{usa1thm1}
Assume \ref{a1}, \eqref{daleq1} and \eqref{spfsa1lm1eq1}.
Then, the trivial solution of \eqref{introeq1} is uniformly stable.
\end{theorem}

\begin{proof}
The proof follows from Theorem~\ref{dalthm1} and Lemma~\ref{spfsa1lm1}.
\end{proof}

\subsection{Global Asymptotic Stability}\label{gasa1}

In this section, we suppose that
\begin{equation}
\int_{t_{0}}^{\infty}\cf(\eta)\Delta\eta=\infty.\label{gasa1eq1}
\end{equation}

\begin{theorem}\label{gasa1thm1}
Assume \ref{a1}, \eqref{daleq1}, \eqref{stla1lm1eq1} and \eqref{gasa1eq1}.
Then, the trivial solution of \eqref{introeq1} is globally asymptotically stable.
\end{theorem}

\begin{proof}
It follows from \eqref{gasa1eq1} together with \ref{alcrl1it1} and \ref{alcrl1it5} of Corollary~\ref{alcrl1} given in \hyperref[appb]{Appendix~B} that $\lim_{t\to\infty}\mathrm{e}_{\ominus(\lambda\cf)}(t,s)=0$ for any $s\in[t_{0},\infty)_{\T}$ and any $\lambda\in(0,1)_{\R}$.
Thus, the proof follows from Theorem~\ref{dalthm2} and Lemma~\ref{spfsa1lm2}.
\end{proof}

\subsection{Uniform Exponential Stability}\label{uesa1}

\begin{theorem}\label{uesa1thm1}
Assume \ref{a1}, \eqref{daleq1}, \eqref{daleq2} and \eqref{stla1lm1eq1}.
Moreover, assume that there exist $M_{1},\lambda_{1}\in\R^{+}$ such that
\begin{equation}
\mathrm{e}_{\ominus(\lambda_{0}\cf)}(t,s)\leq{}M_{1}\mathrm{e}_{\ominus\lambda_{1}}(t,s)\quad\text{for all}\
(t,s)\in\Lambda_{t_{0}},\notag
\end{equation}
where $\Lambda_{t_{0}}$ is defined in \eqref{dalthm1it2eq2} and $\lambda_{0}\in(0,1)_{\R}$ is provided in 
Lemma~\ref{stla1lm1}.
Then, the trivial solution of \eqref{introeq1} is uniformly exponentially stable.
\end{theorem}

\begin{proof}
The proof follows from Theorem~\ref{dalthm3} and Lemma~\ref{spfsa1lm2}.
\end{proof}

\begin{corollary}\label{uesa1crl1}
Assume \ref{a1}, \eqref{daleq1}, \eqref{daleq2} and \eqref{stla1lm1eq1}.
Moreover, assume for every $\lambda\in(0,1)_{\R}$ that
there exist $M_{1},\lambda_{1}\in\R^{+}$ (which may depend on $\lambda$) such that
\begin{equation}
\mathrm{e}_{\ominus(\lambda\cf)}(t,s)\leq{}M_{1}\mathrm{e}_{\ominus\lambda_{1}}(t,s)\quad\text{for all}\ (t,s)\in\Lambda_{t_{0}},\notag
\end{equation}
where $\Lambda_{t_{0}}$ is defined in \eqref{dalthm1it2eq2}.
Then, the trivial solution of \eqref{introeq1} is uniformly exponentially stable.
\end{corollary}

\section{Stability Results under {\protect\ref{a2}}}\label{srua2}

This section includes analogous results
to those in Section~\ref{srua1} under the condition \ref{a2}.
We will be relaxing the conditions \eqref{stla1lm1eq1} and \eqref{spfsa1lm1eq1} of the previous section
by replacing the condition \ref{a1} with the stronger one \ref{a2}.
We will show that the condition \ref{a2} for $\dly$ implies the same for $\dly_{\ast}$.
Indeed, under \ref{a2}, we have
\begin{equation}
\dly_{\ast}(\sigma(t))\leq\dly(\sigma(t))\leq{}t
\quad\text{for all}\ t\in[t_{0},\infty)_{\T}.\notag
\end{equation}

\subsection{A Technical Lemma}\label{stla2}

\begin{lemma}\label{stla2lm1}
Assume \ref{a2} and
\begin{equation}
\sup_{t\in[t_{0},\infty)_{\T}}\negmedspace\bigg\{\negmedspace\int_{\dly_{\ast}(t)}^{t}\cf(\eta)\Delta\eta\bigg\}<1.\label{stla2lm1eq1}
\end{equation}
Then, $-\cf\in\reg{+}([t_{0},\infty)_{\T},\R)$ and there exists $\lambda_{0}\in(0,1)_{\R}$ such that
\begin{equation}
\int_{\dly_{\ast}(t)}^{t}\mathrm{e}_{\lambda_{0}(\ominus(-\cf))}(\sigma(t),\dly_{\ast}(\eta))\cf(\eta)\Delta\eta<1-\lambda_{0}\quad\text{for all}\ t\in[t_{0},\infty)_{\T}.\notag
\end{equation}
\end{lemma}

\begin{proof}
From \eqref{stla2lm1eq1}, there exists $\nu_{0}\in(0,1)_{\R}$ such that
\begin{equation}
\int_{\dly_{\ast}(t)}^{t}\cf(\eta)\Delta\eta<\nu_{0}\quad\text{for all}\ t\in[t_{0},\infty)_{\T}.\label{stla2lm1prfeq1}
\end{equation}
First, let us prove that $-\cf\in\reg{+}([t_{0},\infty)_{\T},\R)$.
By \eqref{stla2lm1prfeq1}, we obtain that
\begin{equation}
\nu_{0}>\int_{\dly_{\ast}(\sigma(t))}^{\sigma(t)}\cf(\eta)\Delta\eta
=\cf(t)\mu(t)+\int_{\dly_{\ast}(\sigma(t))}^{t}\cf(\eta)\Delta\eta\quad\text{for all}\ t\in[t_{0},\infty)_{\T},\label{stla2lm1prfeq2}
\end{equation}
which yields
\begin{equation}
1-\cf(t)\mu(t)>1-\nu_{0}>0\quad\text{for all}\ t\in[t_{0},\infty)_{\T}.\label{stla2lm1prfeq3}
\end{equation}
Therefore, $-\cf\in\reg{+}([t_{0},\infty)_{\T},\R)$.
Using \eqref{stla2lm1prfeq1} and \eqref{stla2lm1prfeq2}, we see that
\begin{equation}
\int_{\dly_{\ast}(t)}^{\sigma(t)}\cf(\eta)\Delta\eta=\int_{\dly_{\ast}(t)}^{t}\cf(\eta)\Delta\eta+\cf(t)\mu(t)
<2\nu_{0}\quad\text{for all}\ t\in[t_{0},\infty)_{\T}.\label{stla2lm1prfeq4}
\end{equation}
Using \eqref{stla2lm1prfeq3}, we have
\begin{equation}
\cf(t)\leq\frac{\cf(t)}{1-\cf(t)\mu(t)}=\big(\ominus(-\cf)\big)(t)<\frac{1}{1-\nu_{0}}\cf(t)\quad\text{for all}\ t\in[t_{0},\infty)_{\T}.\label{stla2lm1prfeq5}
\end{equation}
Note that $-(1-\lambda)\cf\in\reg{+}([t_{0},\infty)_{\T},\R)$ for $\lambda\in(0,1)_{\R}$.
From \eqref{stla2lm1prfeq5}, we get the estimate
\begin{align}
\int_{\dly_{\ast}(t)}^{t}\mathrm{e}_{\lambda(\ominus(-\cf))}(\sigma(t),\dly_{\ast}(\eta))\cf(\eta)\Delta\eta
\leq&\mathrm{e}_{\lambda(\ominus(-\cf))}\big(\sigma(t),\dly_{\ast}^{2}(t)\big)\int_{\dly_{\ast}(t)}^{t}\cf(\eta)\Delta\eta\notag\\
<&\mathrm{e}_{\frac{\lambda}{1-\nu_{0}}\cf}\big(\sigma(t),\dly_{\ast}^{2}(t)\big)\int_{\dly_{\ast}(t)}^{t}\cf(\eta)\Delta\eta\notag
\end{align}
for all $\lambda\in(0,1)_{\R}$ and all $t\in[t_{0},\infty)_{\T}$.
This yields by using \eqref{stla2lm1prfeq1} and \eqref{stla2lm1prfeq4} that
\begin{equation}
\int_{\dly_{\ast}(t)}^{t}\mathrm{e}_{\lambda(\ominus(-\cf))}(\sigma(t),\dly_{\ast}(\eta))\cf(\eta)\Delta\eta
<\nu_{0}\exp\bigg\{\frac{3\lambda\nu_{0}}{1-\nu_{0}}\bigg\}\notag
\end{equation}
for all $\lambda\in(0,1)_{\R}$ and all $t\in[t_{0},\infty)_{\T}$.
Now, consider the function $\varphi\in\cnt{}([0,1]_{\R},\R)$ defined by
\begin{equation}
\phi(\lambda)=(1-\lambda)-\nu_{0}\exp\bigg\{\frac{3\lambda\nu_{0}}{1-\nu_{0}}\bigg\}\quad\text{for}\ \lambda\in[0,1]_{\R}.\notag
\end{equation}
Clearly, $\phi(1-\nu_{0})=\nu_{0}(1-\mathrm{e}^{3\nu_{0}})<0$ and $\phi(0)=1-\nu_{0}>0$.
Therefore, we may find $\lambda_{0}\in(0,1-\nu_{0})_{\R}$ such that $\phi(\lambda_{0})=0$, i.e.,
\begin{align}
\int_{\dly_{\ast}(t)}^{t}\mathrm{e}_{\lambda_{0}(\ominus(-\cf))}(\sigma(t),\dly_{\ast}(\eta))\cf(\eta)\Delta\eta
<&\nu_{0}\exp\bigg\{\frac{3\lambda_{0}\nu_{0}}{1-\nu_{0}}\bigg\}
=(1-\lambda_{0})-\phi(\lambda_{0})\notag\\
=&1-\lambda_{0}\notag
\end{align}
for all $t\in[t_{0},\infty)_{\T}$,
which concludes the proof.
\end{proof}

\subsection{Some Properties of the Fundamental Solution}\label{spfsa2}

\begin{lemma}\label{spfsa2lm1}
Assume \ref{a2} and
\begin{equation}
\int_{\dly(t)}^{t}\cf(\eta)\Delta\eta\leq1
\quad\text{for all}\ \ \ \ t\in[t_{0},\infty)_{\T}.\label{spfsa2lm1eq1}
\end{equation}
Then, \eqref{spfsa1lm1eq2} holds.
\end{lemma}

\begin{proof}
Fix $s\in[t_{0},\infty)_{\T}$ and denote $x(t):=\fun(t,s)$ for $t\in[s,\infty)_{\T}$.
First, let $x$ be positive on $[s,\infty)_{\T}$,
then $x$ is decreasing on $[s,\infty)_{\T}$,
which implies $0<x(t)\leq1$ for all $t\in[s,\infty)_{\T}$.
Thus the claim is true for positive $x$.
Next, let $x$ have some generalized zeros on $[s,\infty)_{\T}$, i.e.,
there exists $t_{1}\in[s,\infty)_{\T}$ such that either $x(t_{1})=0$ or $x(t_{1})>0$ and $x^{\sigma}(t_{1})<0$.
Hence, $0\leq{}x(t)\leq1$ for all $t\in[s,t_{1}]_{\T}$.
Let
\begin{equation}
t_{2}:=\sup\{t\in[s,\infty)_{\T}:\ |x(\eta)|\leq1\ \text{for all}\ \eta\in[s,t]_{\T}\}.\notag
\end{equation}
Clearly, $t_{2}\geq{}t_{1}$.
To prove $t_{2}=\infty$, assume the contrary that $t_{2}$ is finite.
Assume for now that $t_{2}$ is right-scattered, i.e., $\mu(t_{2})>0$.
Then, we have $|x^{\sigma}(t_{2})|>1$ and $|x(t)|\leq1$ for all $t\in[s,t_{2}]_{\T}$.
Without loss of generality, let $x^{\sigma}(t_{2})>1$
(he case where $x^{\sigma}(t_{2})<1$ is treated similarly),
which implies $x^{\Delta}(t_{2})>0$.
From \eqref{introeq1}, we have $x(\dly(t_{2}))<0$.
Note that
\begin{equation}
\mu(t_{2})\cf(t_{2})=\int_{t_{2}}^{\sigma(t_{2})}\cf(\eta)\Delta\eta
\leq\int_{\dly(\sigma(t_{2}))}^{\sigma(t_{2})}\cf(\eta)\Delta\eta\leq1,\notag
\end{equation}
where we have used \ref{a2} for the first inequality and \eqref{spfsa2lm1eq1} in the last inequality.
Integrating \eqref{introeq1} from $\alpha(t_{2})$ to $\sigma(t_{2})$, we get
\begin{align}
x^{\sigma}(t_{2})
=&x(\dly(t_{2}))-\int_{\dly(t_{2})}^{\sigma(t_{2})}\cf(\eta)x(\dly(\eta))\Delta\eta\notag\\
=&[1-\mu(t_{2})\cf(t_{2})]x(\dly(t_{2}))-\int_{\dly(t_{2})}^{t_{2}}\cf(\eta)x(\dly(\eta))\Delta\eta\notag\\
\leq&-\int_{\dly(t_{2})}^{t_{2}}\cf(\eta)x(\dly(\eta))\Delta\eta\leq\int_{\dly(t_{2})}^{t_{2}}\cf(\eta)\Delta\eta\leq1,\notag
\end{align}
which contradicts $x^{\sigma}(t_{2})>1$.
This shows that $t_{2}$ is right-dense.
That is, $x$ is continuous at $t_{2}$.
In this case, $|x(t_{2})|=1$.
Hence, we can find $t_{3}\in(t_{2},\infty)_{\T}$
such that $|x(t_{3})|>1$ and $x$ is of fixed sign on $[t_{2},t_{3}]_{\T}$.
Without loss of generality, assume that $x(t_{2})=1$ and $x(t)>0$ for all $t\in[t_{2},t_{3}]_{\T}$.
The case where $x(t_{2})=-1$ and $x(t)<0$ for all $t\in[t_{2},t_{3}]_{\T}$ is treated similarly.
Let $t_{4}$ be the greatest generalized zero of $x$ in $[s,t_{2})_{\T}$.
Hence, we have either $x(t_{4})=0$ or $x(t_{4})<0$ and $x^{\sigma}(t_{4})>0$.
Further, $x(t)>0$ for all $t\in(t_{4},t_{3}]_{\T}$.
We can also find $t_{5}\in[t_{2},t_{3}]_{\T}$ such that $x^{\Delta}(t_{5})>0$ and $x^{\sigma}(t_{5})>1$.
This implies $x(\dly(t_{5}))<0$ by \eqref{introeq1}.
If $t_{4}<\dly(t_{5})\leq{}t_{5}$,
then $x(\dly(t_{5}))>0$,
which is a contradiction.
Thus, $\dly(t_{5})\leq{}t_{4}$.
Further, we have $|x(\dly(t))|\leq1$ for all $t\in[t_{4},t_{5}]_{\T}$.
So, integrating \eqref{introeq1} from $t_{4}$ to $\sigma(t_{5})$ yields
\begin{align}
x^{\sigma}(t_{5})=&x(t_{4})-\int_{t_{4}}^{\sigma(t_{5})}\cf(\eta)x(\dly(\eta))\Delta\eta\notag\\
=&x(t_{4})-\mu(t_{5})\cf(t_{5})x(\dly(t_{5}))-\int_{t_{4}}^{t_{5}}\cf(\eta)x(\dly(\eta))\Delta\eta\notag\\
\leq&\int_{t_{4}}^{t_{5}}\cf(\eta)\Delta\eta
\leq\int_{\dly(t_{5})}^{t_{5}}\cf(\eta)\Delta\eta\leq1,\notag
\end{align}
which contradicts $x^{\sigma}(t_{5})>1$.
This implies $t_{2}=\infty$ and completes the proof.
\end{proof}

\begin{lemma}\label{spfsa2lm2}
Assume \ref{a2}, \eqref{daleq1} and \eqref{stla2lm1eq1}.
Then,
\begin{equation}
|\fun(t,s)|\leq\mathrm{e}_{-\lambda_{0}\cf}(t,s)\quad\text{for all}\ (t,s)\in\Lambda_{t_{0}},\notag
\end{equation}
where $\Lambda_{t_{0}}$ is defined in \eqref{dalthm1it2eq2} and $\lambda_{0}\in(0,1)_{\R}$ is provided in 
Lemma~\ref{stla2lm1}.
\end{lemma}

\begin{proof}
For simplicity of notation, fix $s\in[t_{0},\infty)_{\T}$ and let $x(t):=\fun(t,s)$ for $t\in[s,\infty)_{\T}$.
We may suppose that the claim of Lemma~\ref{stla2lm1} holds with $\lambda_{0}\in(0,1-\nu_{0})_{\R}$ on $[s,\infty)_{\T}$,
where $\nu_{0}\in(0,1)_{\R}$ satisfies \eqref{stla2lm1prfeq1}.
From \eqref{introeq1}, we have
\begin{align}
x^{\Delta}(t)=&-\cf(t)x(t)+\cf(t)\int_{\dly(t)}^{t}x^{\Delta}(\eta)\Delta\eta\notag\\
=&-\cf(t)x(t)-\cf(t)\int_{\dly(t)}^{t}\cf(\eta)x(\dly(\eta))\Delta\eta\label{spfsa2lm2prfeq1}
\end{align}
for all $t\in[\dly_{-1}(s),\infty)_{\T}$.
Applying the solution representation formula in Lemma~\ref{dallm1}
for \eqref{spfsa2lm2prfeq1}, we get
\begin{equation}
\begin{aligned}[]
x(t)=&x(\dly_{-1}(s))\mathrm{e}_{-\cf}(t,\dly_{-1}(s))
-\int_{\dly_{-1}(s)}^{t}\mathrm{e}_{-\cf}(t,\sigma(\eta))\cf(\eta)\\
&\phantom{x(\dly_{-1}(s))\mathrm{e}_{-\cf}(t,\dly_{-1}(s))-\int_{\dly_{-1}(s)}^{t}}
\times\int_{\dly(\eta)}^{\eta}\cf(\zeta)x(\dly(\zeta))\Delta\zeta\Delta\eta
\end{aligned}\label{spfsa2lm2prfeq2}
\end{equation}
for all $t\in[\dly_{-1}(s),\infty)_{\T}$.
Let us define a function $y\in\crd{1}([s,\infty)_{\T},\R)$ by
\begin{equation}
y(t):=
\begin{cases}
x(t)\mathrm{e}_{\lambda_{0}(\ominus(-\cf))}(t,\dly_{-1}(s)),&t\in[\dly_{-1}(s),\infty)_{\T},\\
x(t),&t\in[s,\dly_{-1}(s)]_{\T}.
\end{cases}\label{spfsa2lm2prfeq3}
\end{equation}
Multiplying \eqref{spfsa2lm2prfeq2} by $\mathrm{e}_{\lambda_{0}(\ominus(-\cf))}(\cdot,\dly_{-1}(s))$ and considering the fact that
\begin{equation}
\begin{aligned}[]
\lambda_{0}\Big(\ominus\big(-\cf(t)\big)\Big)\oplus\big(-\cf(t)\big)
=&\frac{\lambda_{0}\cf(t)}{1-\cf(t)\mu(t)}-\cf(t)-\frac{\lambda_{0}\big(\cf(t)\big)^{2}\mu(t)}{1-\cf(t)\mu(t)}\\
=&-(1-\lambda_{0})\cf(t)
\end{aligned}\notag
\end{equation}
for all $t\in[\dly_{-1}(s),\infty)_{\T}$, we get
\begin{align}
\begin{split}
y(t)=&y(\dly_{-1}(s))\mathrm{e}_{-(1-\lambda_{0})\cf}(t,\dly_{-1}(s))\\
&-\int_{\dly_{-1}(s)}^{t}\mathrm{e}_{-(1-\lambda_{0})\cf}(t,\sigma(\eta))\cf(\eta)\\
&\phantom{-\int_{\dly_{-1}(s)}^{t}}\times\int_{\dly(\eta)}^{\eta}\mathrm{e}_{-(1-\lambda_{0})\cf}(\sigma(\eta),\dly_{-1}(s))\cf(\zeta)x(\dly(\zeta))\Delta\zeta\Delta\eta
\end{split}\notag\\
\begin{split}
=&y(\dly_{-1}(s))\mathrm{e}_{-(1-\lambda_{0})\cf}(t,\dly_{-1}(s))\\
&-\int_{\dly_{-1}(s)}^{t}\mathrm{e}_{-(1-\lambda_{0})\cf}(t,\sigma(\eta))\cf(\eta)\\
&\phantom{-\int_{\dly_{-1}(s)}^{t}}\times\int_{\dly(\eta)}^{\eta}\mathrm{e}_{\lambda_{0}(\ominus(-\cf))}(\sigma(\eta),\dly(\zeta))\cf(\zeta)y(\dly(\zeta))\Delta\zeta\Delta\eta
\end{split}\label{spfsa2lm2prfeq4}
\end{align}
for all $t\in[\dly_{-1}(s),\infty)_{\T}$.

Next, we claim that $|y(t)|\leq1$ for all $t\in[s,\infty)_{\T}$.
Let
$$t_{2}:=\sup\{t\in[s,\infty)_{\T}:\ |y(\eta)|\leq1\ \text{for all}\ \eta\in[s,t]_{\T}\}.$$
To prove $t_{2}=\infty$,
assume the contrary that $t_{2}$ is finite.
Clearly, $t_{2}>\dly_{-1}(s)$ by \eqref{spfsa2lm2prfeq3} and Lemma~\ref{spfsa2lm1}.
Assume for now that $t_{2}$ is right-scattered.
Then, we have $|y^{\sigma}(t_{2})|>1$ and $|y(t)|\leq1$ for all $t\in[s,t_{2}]_{\T}$.
Using \eqref{spfsa2lm2prfeq4}, Lemma~\ref{stla2lm1} and \cite[Theorem~2.39]{MR1843232}, we obtain
\begin{align}
\begin{split}
|y^{\sigma}(t_{2})|\leq&\mathrm{e}_{-(1-\lambda_{0})\cf}(\sigma(t_{2}),\dly_{-1}(s))\\
&+\int_{\dly_{-1}(s)}^{\sigma(t_{2})}\mathrm{e}_{-(1-\lambda_{0})\cf}(\sigma(t_{2}),\eta)\cf(\eta)\\
&\phantom{\int_{\dly_{-1}(s)}^{\sigma(t_{2})}}
\times\int_{\dly_{\ast}(\eta)}^{\eta}\mathrm{e}_{\lambda_{0}(\ominus(-\cf))}(\sigma(\eta),\dly_{\ast}(\zeta))\cf(\zeta)\Delta\zeta\Delta\eta
\end{split}\notag\\
\begin{split}
<&\mathrm{e}_{-(1-\lambda_{0})\cf}(\sigma(t_{2}),\dly_{-1}(s))\\
&+(1-\lambda_{0})\int_{\dly_{-1}(s)}^{\sigma(t_{2})}\mathrm{e}_{-(1-\lambda_{0})\cf}(\sigma(t_{2}),\sigma(\eta))\cf(\eta)\Delta\eta
\end{split}\notag\\
=&\mathrm{e}_{-(1-\lambda_{0})\cf}(\sigma(t_{2}),\dly_{-1}(s))
+\int_{\dly_{-1}(s)}^{\sigma(t_{2})}\frac{\Delta}{\Delta\eta}\mathrm{e}_{-(1-\lambda_{0})\cf}(\sigma(t_{2}),\eta)\Delta\eta=1,\notag
\end{align}
which is a contradiction.
Thus $t_{2}$ is right-dense, hence $y$ is continuous at $t_{2}$.
In this case, $|y(t_{2})|=1$.
Using Lemma~\ref{stla1lm1} and \eqref{spfsa1lm2prfeq4},
we can proceed as above and show that $|y(t_{2})|<1$,
which is also a contradiction.
This implies $t_{2}=\infty$, i.e.,
$|y(t)|\leq1$ for all $t\in[s,\infty)_{\T}$.

It follows from \eqref{spfsa2lm2prfeq3} that
\begin{equation}
|y(t)|=|x(t)|\mathrm{e}_{\lambda_{0}(\ominus(-\cf))}(t,\dly_{-1}(s))
\geq|x(t)|\mathrm{e}_{\ominus(-\lambda_{0}\cf)}(t,\dly_{-1}(s))
\label{spfsa2lm2prfeq5}
\end{equation}
for all $t\in[\dly_{-1}(s),\infty)_{\T}$
since we have
\begin{equation}
\lambda_{0}\Big(\ominus\big(-\cf(t)\big)\Big)
=\frac{\lambda_{0}\cf(t)}{1-\cf(t)\mu(t)}
\geq\frac{\lambda_{0}\cf(t)}{1-\lambda_{0}\cf(t)\mu(t)}
=\ominus\big(-\lambda_{0}\cf(t)\big)
\notag
\end{equation}
for all $t\in[\dly_{-1}(s),\infty)_{\T}$.
From \eqref{spfsa2lm2prfeq5}, we obtain
\begin{equation}
|x(t)|\leq\mathrm{e}_{-\lambda_{0}\cf}(t,\dly_{-1}(s))=\mathrm{e}_{\ominus(-\lambda_{0}\cf)}(\dly_{-1}(s),s)\mathrm{e}_{-\lambda_{0}\cf}(t,s)
\label{spfsa2lm2prfeq6}
\end{equation}
for all $t\in[\dly_{-1}(s),\infty)_{\T}$.
As in the proof of Lemma~\ref{stla2lm1} and by virtue of \cite[Lemma~3.2]{MR2842561}, we estimate
\begin{align}
\mathrm{e}_{\ominus(-\lambda_{0}\cf)}(\dly_{-1}(s),s)
=&\mathrm{e}_{\frac{\lambda_{0}\cf}{1-\lambda_{0}\nu_{0}}}(\dly_{-1}(s),s)
\leq\exp\bigg\{\frac{\lambda_{0}}{1-\lambda_{0}\nu_{0}}\int_{s}^{\dly_{-1}(s)}\cf(\eta)\Delta\eta\bigg\}\notag\\
\leq&\mathrm{e}^{\frac{\lambda_{0}K_{0}}{1-\lambda_{0}\nu_{0}}}=:M_{0},\notag
\end{align}
where $K_{0}:=\sup_{s\in[t_{0},\infty)_{\T}}\negmedspace\big\{\negmedspace\int_{s}^{\dly_{-1}(s)}\cf(\eta)\Delta\eta\big\}$.
Hence, by \eqref{spfsa2lm2prfeq6}, $M_{0}>1$ and the fact that $|x(t)|\leq1$ for all $t\in[s,\dly_{-1}(s)]_{\T}$,
we have $|x(t)|\leq{}M_{0}\mathrm{e}_{-\lambda_{0}\cf}(t,s)$ for all $t\in[s,\infty)_{\T}$,
which concludes the proof.
\end{proof}

\subsection{Uniform Stability}\label{usa2}

\begin{theorem}\label{usa2thm1}
Assume \ref{a2}, \eqref{daleq1} and \eqref{spfsa2lm1eq1}.
Then, the trivial solution of \eqref{introeq1} is uniformly stable.
\end{theorem}

\begin{proof}
The proof follows from Theorem~\ref{dalthm1} and Lemma~\ref{spfsa2lm1}.
\end{proof}

\subsection{Global Asymptotic Stability}\label{gasa2}

In this section, we suppose that \eqref{gasa1eq1} holds.

\begin{theorem}\label{gasa2thm1}
Assume \ref{a2}, \eqref{daleq1}, \eqref{gasa1eq1} and \eqref{stla2lm1eq1}.
Then, the trivial solution of \eqref{introeq1} is globally asymptotically stable.
\end{theorem}

\begin{proof}
The proof follows from Theorem~\ref{dalthm2} and Lemma~\ref{spfsa2lm2} together with \ref{alcrl1it1} and \ref{alcrl1it6} of Corollary~\ref{alcrl1} given in \hyperref[appb]{Appendix~B}.
\end{proof}

\subsection{Uniform Exponential Stability}\label{uesa2}

\begin{theorem}\label{uesa2thm1}
Assume \ref{a2}, \eqref{daleq1}, \eqref{daleq2} and \eqref{stla2lm1eq1}.
Moreover, assume that there exist $M_{1},\lambda_{1}\in\R^{+}$ such that
\begin{equation}
\mathrm{e}_{-\lambda_{0}\cf}(t,s)\leq{}M_{1}\mathrm{e}_{\ominus\lambda_{1}}(t,s)\quad\text{for all}\ (t,s)\in\Lambda_{t_{0}},\notag
\end{equation}
where $\Lambda_{t_{0}}$ is defined in \eqref{dalthm1it2eq2} and $\lambda_{0}\in(0,1)_{\R}$ is provided in 
Lemma~\ref{spfsa2lm2}.
Then, the trivial solution of \eqref{introeq1} is uniformly exponentially stable.
\end{theorem}

\begin{proof}
The proof follows from Theorem~\ref{dalthm3} and Lemma~\ref{spfsa2lm2}.
\end{proof}

\begin{corollary}\label{uesa2crl1}
Assume \ref{a2}, \eqref{daleq1}, \eqref{daleq2} and \eqref{stla2lm1eq1}.
Moreover, assume that for every $\lambda\in(0,1)_{\R}$ there exist $M_{1},\lambda_{1}\in\R^{+}$ (which may depend on $\lambda$) such that
\begin{equation}
\mathrm{e}_{-\lambda\cf}(t,s)\leq{}M_{1}\mathrm{e}_{\ominus\lambda_{1}}(t,s)\quad\text{for all}\ (t,s)\in\Lambda_{t_{0}},\label{uesa2crl1eq1}
\end{equation}
where $\Lambda_{t_{0}}$ is defined in \eqref{dalthm1it2eq2}.
Then, the trivial solution of \eqref{introeq1} is uniformly exponentially stable.
\end{corollary}

\section{Some Applications}\label{sapp}

This section includes three examples,
which show that our results are easily applicable and fill some gaps in the literature.

Before presenting our examples,
we would like to make a remark.

\begin{remark}\label{sappremk1}
If the delay function $\dly$ is increasing, then $\dly_{\ast}=\dly$ (which also holds for nondecreasing $\dly$) and $\dly_{-1}=\dly^{-1}$,
where $\dly^{-1}$ is the inverse of the delay $\dly$.
So, \eqref{daleq1} is satisfied by \eqref{spfsa2lm1eq1}, which is implied by \eqref{stla2lm1eq1}.
It should also be noted that \eqref{spfsa1lm1eq1} implies \eqref{spfsa2lm1eq1}.
Further, in this case, \eqref{daleq2} is equivalent to
\begin{equation}
\limsup_{t\to\infty}[t-\dly(t)]<\infty.\label{sappremk1eq1}
\end{equation}
\end{remark}

In the examples below, the delay function $\dly$ is strictly increasing.
Hence, by Remark~\ref{sappremk1}, we will omit the justification of \eqref{daleq1}.

\begin{example}\label{saex1}
Consider the time scale $\mathbb{P}_{1,1}=\cup_{k\in\Z}[2k,2k+1]_{\R}=\cdots\cup[0,1]_{\R}\cup[2,3]_{\R}\cup\cdots$ and the dynamic equation
\begin{equation}
x^{\Delta}(t)+\cf(t)x(\dly(t))=0\quad\text{for}\ t\in[1,\infty)_{\mathbb{P}_{1,1}},\label{saex1eq1}
\end{equation}
where
\begin{equation}
\cf(t):=
\begin{cases}
a4^{[t]},&t\in[2k,2k+1)_{\R}\ \text{and}\ k\in\N, \\
a,&t=2k+1\ \text{and}\ k\in\N_{0}
\end{cases}
\ \text{and}\
\dly(t):=t-\big(\{t\}(1-\{t\})\big)^{[t]}
\notag
\end{equation}
for $t\in[1,\infty)_{\mathbb{P}_{1,1}}$.
Here, $a\in\R^{+}$, $[\cdot]$ and $\{\cdot\}$ denote the greatest integer and the fractional part, respectively.
Note that $\cf(2k)=a16^{k}$ for $k\in\N$, i.e., the coefficient $\cf$ is unbounded on $[1,\infty)_{\mathbb{P}_{1,1}}$.
More precisely, we have $\dly(n)=n$ for $n\in\N$ and $\dly(\sigma(2k-1))=\dly(2k)=2k$ for $k\in\N$.
Hence, \ref{a1} holds but \ref{a2} is not satisfies.
We evaluate
\begin{equation}
\int_{\dly(t)}^{\sigma(t)}\cf(\eta)\Delta\eta
=
\begin{cases}
a\big(4\{t\}(1-\{t\})\big)^{[t]},&t\in[2k,2k+1)_{\R}\ \text{and}\ k\in\N, \\
a,&t=2k+1\ \text{and}\ k\in\N_{0}
\end{cases}
\notag
\end{equation}
for $t\in[1,\infty)_{\mathbb{P}_{1,1}}$.
Note that
\begin{equation}
\max_{t\in[2n,2n+1]}\big(4\{t\}(1-\{t\})\big)^{[t]}
=\big(4\{t\}(1-\{t\})\big)^{[t]}\bigg|_{t=2n+\frac{1}{2}}
=1
\quad\text{for}\ n\in\N.\notag
\end{equation}
Applying Theorem~\ref{usa1thm1}, we see that the trivial solution of \eqref{saex1eq1} is uniformly stable if $a\leq1$.
If $a<1$, then the trivial solution of \eqref{saex1eq1} is globally asymptotically stable by
Theorem~\ref{gasa1thm1} since \eqref{gasa1eq1} holds readily.
Using Corollary~\ref{uesa1crl1} with $\lambda_{1}:=\lambda{}a$ for $\lambda\in(0,1)_{\R}$,
we see that the trivial solution of \eqref{saex1eq1} is uniformly exponentially stable provided that $a<1$.

Since the delay $\dly$ is not strict, the result in \cite{MR2335381} (see also \cite{MR2437887}) is not applicable.
Further, as the coefficient $\cf$ is unbounded, the results in \cite{MR2927064} are not applicable, either.
\end{example}

\begin{example}\label{saex2}
On the time scale $\T=\cup_{n\in\N}[\sinh(n),\cosh(n)]_{\R}$ (whose graininess is unbounded),
we define
\begin{equation}
\cf(t):=
\begin{cases}
a,&t\in[\sinh(n),\cosh(n))_{\R}\ \text{and}\ n\in\N\\
\cfrac{a}{\sinh(n+1)-\cosh(n)},&t=\cosh(n)\ \text{and}\ n\in\N
\end{cases}\notag
\end{equation}
for $t\in[1,\infty)_{\T}$, where $a\in\R^{+}$,
and
\begin{equation}
\dly(t):=t-\frac{\big(\cosh(n)-t\big)\big(t-\sinh(n)\big)}{\cosh(n)-\sinh(n)}
\quad\text{for}\ t\in[\sinh(n),\cosh(n)]_{\R}\ \text{and}\ n\in\N.\notag
\end{equation}
Obviously, \ref{a1} holds.
However, \ref{a2} does not hold since $\dly(\sigma(\cosh(n)))=\dly(\sinh(n+1))=\sinh(n+1)$ for $n\in\N$.
Consider the dynamic equation
\begin{equation}
x^{\Delta}(t)+\cf(t)x(\dly(t))=0\quad\text{for}\ t\in[\sinh(1),\infty)_{\T}.\label{saex2eq1}
\end{equation}
We compute
\begin{equation}
\int_{\dly(t)}^{\sigma(t)}\cf(\eta)\Delta\eta
=
\begin{cases}
a\big(t-\dly(t)\big),&t\in[\sinh(n),\cosh(n))_{\R}\ \text{and}\ n\in\N, \\
a,&t=\cosh(n)\ \text{and}\ n\in\N
\end{cases}
\notag
\end{equation}
for $t\in[\sinh(1),\infty)_{\T}$.
Further, for $n\in\N$, we have
\begin{equation}
\max_{t\in[\sinh(n),\cosh(n)]_{\R}}\big(t-\dly(t)\big)
=\mathrm{e}^{n}\big(\cosh(n)-t\big)\big(t-\sinh(n)\big)\bigg|_{t=\frac{\mathrm{e}^{n}}{2}}
=\frac{1}{4\mathrm{e}^{n}}<1,\notag
\end{equation}
which tends to zero as $n\to\infty$.
By Theorem~\ref{usa1thm1}, the trivial solution of \eqref{saex2eq1} is uniformly stable if $a\leq1$.
And by Theorem~\ref{gasa1thm1}, we also have global asymptotic stability for the trivial solution if $a<1$.
However, we cannot apply Theorem~\ref{uesa1thm1} to provide uniform exponential stability for the trivial solution.

The delay is not strict since $\dly(\sinh(n))=\sinh(n)$ for $n\in\N$,
which shows that the result in \cite{MR2335381} (see also \cite{MR2437887}) does not apply.
The graininess being unbounded implies that the results in \cite{MR2927064} fail for this equation.
\end{example}

\begin{example}\label{saex3}
On the isolated time scale $\T=\Z\backslash3\Z=\{\cdots,1,2,4,5,\cdots\}$ consider the equation
\begin{equation}
x^{\Delta}(t)+\cf(t)x(\rho^{2}(t))=0\quad\text{for}\ t\in[1,\infty)_{\T},\label{saex3eq1}
\end{equation}
where
\begin{equation}
\cf(t):=
\begin{cases}
a,&t\in3\N_{0}+1,\\
b,&t\in3\N_{0}+2,
\end{cases}\notag
\end{equation}
with $a,b\in\R^{+}$.
Then, we see that \ref{a2} holds for $\dly=\rho^{2}$.
Clearly, we have
\begin{equation}
\int_{\rho^{2}(t)}^{t}\cf(\eta)\Delta\eta=\cf(\rho(t))\mu(\rho(t))+\cf(\rho^{2}(t))\mu(\rho^{2}(t))\equiv{}a+2b\quad\text{for}\ t\in[1,\infty)_{\T}.\notag
\end{equation}
Thus, Theorem~\ref{usa2thm1} provides uniform stability for the trivial solution of \eqref{saex3eq1} when $a+2b\leq1$.
Let us note that \cite[Theorem~1.1]{MR2335381} cannot be applied when $a=\frac{7}{10}$ and $b=\frac{1}{10}$.
Now, we see that
\begin{equation}
\int_{1}^{\infty}\cf(\eta)\Delta\eta
=\lim_{n\to\infty}\sum_{k=0}^{n-1}\int_{3k+1}^{3(k+1)+1}\cf(\eta)\Delta\eta
=\lim_{n\to\infty}n[a+2b]=\infty.\notag
\end{equation}
Therefore, the trivial solution of \eqref{introeq1} is globally asymptotically stable by Theorem~\ref{gasa2thm1} if $a+2b<1$.
With $a=\frac{7}{10}$ and $b=\frac{1}{10}$ satisfying $a+2b<1$, \cite[Theorem~1.2]{MR2335381} also fails.
Finally, let $a+2b<1$, then $a<1$ and $b<\frac{1}{2}$.
We will show that the assumptions of Corollary~\ref{uesa2crl1} hold with $M_{1}:=1$ and $\lambda_{1}:=\min\negmedspace\big\{\frac{\lambda{}a}{1-\lambda{}a},\frac{\lambda{}b}{1-2\lambda{}b}\big\}$ (which is positive) for $\lambda\in(0,1)_{\R}$.
To this end, we will prove that
\begin{equation}
\mathrm{e}_{-\lambda\cf}(\sigma(t),t)
\leq\mathrm{e}_{\ominus\lambda_{1}}(\sigma(t),t)\quad\text{for}\ t\in\T\ \text{and}\ \lambda\in(0,1)_{\R}\label{saex3eq2}
\end{equation}
or equivalently
\begin{equation}
1-\lambda\cf(t)\mu(t)
\leq\frac{1}{1+\lambda_{1}\mu(t)}\quad\text{for}\ t\in\T\ \text{and}\ \lambda\in(0,1)_{\R}.\notag
\end{equation}
Indeed, for $t\in\T$ and $\lambda\in(0,1)_{\R}$, we have
\begin{equation}
1-\lambda\cf(t)\mu(t)
=\frac{1}{1+\frac{\lambda\cf(t)}{1-\lambda\cf(t)\mu(t)}\mu(t)}\\
\leq\frac{1}{1+\lambda_{1}\mu(t)}.\notag
\end{equation}
By the semigroup property (see \cite[Lemma~2.31]{MR1843232}) and $\T$
being isolated, it follows from \eqref{saex3eq2} that
\begin{equation}
\mathrm{e}_{-\lambda\cf}(t,s)
=\prod_{\eta\in[s,t)_{\T}}\mathrm{e}_{-\lambda\cf}(\sigma(\eta),\eta)
\leq\prod_{\eta\in[s,t)_{\T}}\mathrm{e}_{\ominus\lambda_{1}}(\sigma(\eta),\eta)
=\mathrm{e}_{\ominus\lambda_{1}}(t,s)
\notag
\end{equation}
for all $(t,s)\in\Lambda_{1}$,
where $\Lambda$ is defined in \eqref{dalthm1it2eq2}.
Thus, the trivial solution of \eqref{saex3eq1} is uniformly exponentially stable by Corollary~\ref{uesa2crl1} if $a+2b<1$.
It should be mentioned that \cite[Theorem~6.1 and Theorem~6.2]{MR2927064} cannot be applied to \eqref{saex3eq1} when
$a=\frac{3}{8}$ and $b=\frac{1}{4}$, which satisfy $a+2b<1$.
\end{example}

\section{Final Discussion}\label{findis}

Let us start with commenting on the attractivity of the trivial solution of \eqref{introeq1}.
Under anyone of the following conditions,
the trivial solution of \eqref{introeq1} is globally attracting:
\begin{enumerate}[label={(\roman*)},leftmargin={*},ref=(\roman*)]
\item \ref{a1}, \eqref{gasa1eq1} and $\limsup_{t\to\infty}\int_{\dly_{\ast}(t)}^{\sigma(t)}\cf(\eta)\Delta\eta<1$;
\item \ref{a2}, \eqref{gasa1eq1} and $\limsup_{t\to\infty}\int_{\dly_{\ast}(t)}^{t}\cf(\eta)\Delta\eta<1$.
\end{enumerate}

For the asymptotic stability of the differential equation \eqref{introeq2} and the difference equation \eqref{introeq3},
many of the results in the literature consider the so-called constant delays, i.e.,
delays of the form $\alpha(t)=t-\alpha_{0}$ for $t\in[t_{0},\infty)_{\T}$,
where $\alpha_{0}$ is a positive constant.
Thus, by Remark~\ref{sappremk1}, the condition \eqref{daleq2} or equivalently the condition \eqref{sappremk1eq1} are satisfied,
see \cite{MR0695252,MR1233677,MR0891356} and
\cite{MR1403453,MR1350434,MR1636333,MR2814561,MR1305480,MR1103855,MR1666138,MR1340734}.
For the asymptotic stability of dynamic equations, we can refer to \cite{MR2335381}, which assumes
\eqref{sappremk1eq1}.
So, this can make an  impression that \eqref{daleq2} is necessary for asymptotic stability.
Our results (Theorems~\ref{gasa1thm1} and \ref{gasa2thm1}) show that the condition \eqref{daleq2}
is not required for the asymptotic stability of \eqref{introeq1}.

As a general example, consider the so-called pantograph equation
\begin{equation}
x^{\prime}(t)+\cf(t)x(\theta{}t)=0\quad\text{for}\ t\in[1,\infty)_{\R},\label{findiseq1}
\end{equation}
where $\cf$ is a continuous function and $\theta\in(0,1)_{\R}$,
which does not satisfy \eqref{sappremk1eq1} (see Remark~\ref{sappremk1}).
Using the idea in \cite{MR1120357} with $u=\ln(t)$,
we can transform \eqref{findiseq1} into
\begin{equation}
y^{\prime}(u)+\mathrm{e}^{u}\cf(\mathrm{e}^{u})y(u-\ln(\tfrac{1}{\theta}))=0\quad\text{for}\ u\in[0,\infty)_{\R},\label{findiseq2}
\end{equation}
where $y^{\prime}(u)$ denotes the derivative of $y$ with respect to $u$ here.
Obviously, \eqref{findiseq2} fulfils \eqref{sappremk1eq1}.
For instance, assume that
\begin{equation}
\sup_{u\geq0}\negmedspace\bigg\{\negmedspace\int_{u-\ln(\tfrac{1}{\theta})}^{u}\mathrm{e}^{\zeta}\cf(\mathrm{e}^{\zeta})\mathrm{d}\zeta\bigg\}<1
\quad\text{and}\quad
\inf_{u\geq0}\big\{\mathrm{e}^{u}\cf(\mathrm{e}^{u})\big\}>0,\notag
\end{equation}
which holds for $\cf(t):=\frac{a}{t}$ for $t\in[1,\infty)_{\R}$,
where $a\in\R^{+}$,
such that $a\ln(\tfrac{1}{\theta})<1$.
By Corollary~\ref{uesa1crl1}, we can say for the fundamental solution $\mathcal{Y}$ of \eqref{findiseq2} that
$|\mathcal{Y}(u,v)|\leq{}M\mathrm{e}^{-\lambda(u-v)}$ for all $u\geq{}v\geq0$,
where $M,\lambda\in\R^{+}$.
Therefore, we see for the fundamental solution $\fun$ of \eqref{findiseq1} that
\begin{equation}
|\fun(t,s)|=\big|\mathcal{Y}\big(\negthinspace\ln(t),\ln(s)\big)\big|
\leq{}M\bigg(\frac{t}{s}\bigg)^{-\lambda}
\quad\text{for all}\ t\geq{}s\geq1.\notag
\end{equation}
This shows that the trivial solution of the pantograph equation \eqref{findiseq1}
is globally asymptotically stable by Theorem~\ref{dalthm2} (cf \cite[Theorem~2.6]{MR3418565}).

Consider the two-term equation with both a delay and a non-delay term
\begin{equation}
x^{\Delta}(t)+A(t)x^{\sigma}(t)+B(t)x(\beta(t))=0\quad\text{for}\ t\in[t_{0},\infty)_{\T},\label{findiseq3}
\end{equation}
where $A,B\in\crd{}([t_{0},\infty)_{\T},\R_{0}^{+})$ and $\beta\in\crd{}([t_{0},\infty)_{\T},\T)$ satisfies $\beta(t)\leq{}t$ for all $t\in[t_{0},\infty)_{\T}$ with $\lim_{t\to\infty}\beta(t)=\infty$.
The substitution $y(t):=\mathrm{e}_{A}(t,t_{0})x(t)$ for $t\in[t_{0},\infty)_{\T}$ transforms \eqref{findiseq3} into the single-term equation
\begin{equation}
y^{\Delta}(t)+B(t)\mathrm{e}_{A}(t,\beta(t))y(\beta(t))=0\quad\text{for}\ t\in[t_{0},\infty)_{\T}.\label{findiseq4}
\end{equation}
By virtue of Lemma~\ref{allm4},
the stability of \eqref{findiseq3} is equivalent to that of \eqref{findiseq4} under the condition that
\begin{equation}
\int_{t_{0}}^{\infty}A(\eta)\Delta\eta<\infty.\label{findiseq5}
\end{equation}
By using Corollary~\ref{alcrl1}, similar idea can be applied to show equivalence of stability of
\begin{equation}
x^{\Delta}(t)+A(t)x(t)+B(t)x(\beta(t))=0\quad\text{for}\ t\in[t_{0},\infty)_{\T}\notag
\end{equation}
and
\begin{equation}
y^{\Delta}(t)+B(t)\mathrm{e}_{\ominus(-A)}(\sigma(t),\beta(t))y(\beta(t))=0\quad\text{for}\ t\in[t_{0},\infty)_{\T}\notag
\end{equation}
provided that $-A\in\reg{+}(\T,\R)$ and \eqref{findiseq5} holds.
This discussion also brings an exponential estimate for the solutions of two-term equations \eqref{findiseq3}.
For instance, if the solution $y$ of \eqref{findiseq4} satisfies $|y(t)|\leq{}M$ for $t\in[t_{0},\infty)_{\T}$, where $M\in\R^{+}$,
then the corresponding solution $x$ of \eqref{findiseq3} satisfies $|x(t)|\leq{}M\mathrm{e}_{\ominus{}A}(t,t_{0})$ for $t\in[t_{0},\infty)_{\T}$ (see the proofs of Lemma~\ref{spfsa1lm2} and Lemma~\ref{spfsa2lm2}).

Let $\{t_{n}\}_{n\in\N_{0}}$ be an increasing unbounded sequence of reals
and consider on the time scale $\T:=\{t_{n}\}_{n\in\N_{0}}$, the equation
\begin{equation}
x^{\Delta}(t)+\frac{a}{\mu(t)}x(\rho(t))=0\quad\text{for}\ t\in[t_{1},\infty)_{\T},\label{findiseq6}
\end{equation}
where $a\in\R^{+}$.
Here, the delay function $\alpha$ is the backward jump operator $\rho$.
Using the so-called simple useful formula, we get
\begin{equation}
x(\sigma(t))=x(t)-ax(\rho(t))\quad\text{for}\ t\in[t_{1},\infty)_{\T},\notag
\end{equation}
or equivalently
\begin{equation}
x(t_{n+1})=x(t_{n})-ax(t_{n-1})\quad\text{for}\ n=1,2,\cdots.\notag
\end{equation}
This leads to the first-order vector recurrence
\begin{equation}
\left(
\begin{array}{c}
  x(t_{n+1}) \\
  x(t_{n})
\end{array}
\right)
=
\left(
  \begin{array}{rr}
    1 & -a \\
    1 & 0
  \end{array}
\right)
\left(
\begin{array}{c}
  x(t_{n}) \\
  x(t_{n-1})
\end{array}
\right)
\quad\text{for}\ n=1,2,\cdots,\notag
\end{equation}
where the coefficient matrix has the eigenvalues $\frac{1}{2}\big(1\pm\sqrt{1-4a}\big)$,
which are less than or equal to $1$ in absolute value if and only if $a\leq1$.
Hence, the trivial solution of \eqref{findiseq6} is uniformly stable if and only if $a\leq1$.

Clearly, \ref{a2} holds with equality for \eqref{findiseq6} since for isolated time scales $\sigma$ and $\rho$ are the inverses of each other on $(\inf\T,\sup\T)_{\T}$ (see \cite[Example~1.4]{MR1843232}).
Further, the condition \eqref{spfsa2lm1eq1} turns out to be $a\leq1$.
This discussion shows that the conditions
of Theorem~\ref{usa2thm1} (also of Theorem~\ref{gasa2thm1}) are sharp.

!!!!!!!!!
The results of the present paper can be viewed as the generalization of the classical resuts of \cite{MR0508721} for delay 
differential equations and recent investigation \cite{MR2963461} for  delay difference equations on the relations of the 
fundamental function (the cauchy operator) and various stability types to  delay dynamic equations. 

In \cite[\S~6.6]{MR0508721}, linear differential systems with distributed delays are considered.
Lemmas~6.2 and 6.3 presented therein are similar to Theorems~\ref{dalthm1} and \ref{dalthm3} for delay dynamic equations.
More precisely, uniform stability and uniform exponential stability are related with boundedness 
and exponential decay of the fundamental solution, respectively.
Further, some other properties (using the uniform boundedness principle) of the solution operator are proved.
In this direction, if we  define the Cauchy operator $\mathcal{C}$ by
\begin{equation}
(\mathcal{C}\varphi)(t):=\fun(t,s)x_{0}-\int_{s}^{t}\fun(t,\sigma(\eta))\cf(\eta)
\varphi(\dly(\eta))\Delta\eta\quad\text{for}\ t\in[s,\infty)_{\T},\notag
\end{equation}
then by using the technique in \cite[\S~6.6]{MR0508721},
we can show that the properties of the Cauchy operator $\mathcal{C}$ are aligned with the asymptotics of the fundamental 
solution $\fun$.
For instance, the fundamental solution $\fun$ is bounded if and only if the Cauchy operator $\mathcal{C}$ is bounded,
or the fundamental solution $\fun$ satisfies an exponential estimate if and only if the Cauchy operator 
$\mathcal{C}$ satisfies an exponential estimate.

On the other hand, in \cite{MR2963461}, Kulikov and Malygina studied various stability types of linear difference equations,
and related stability types with the certain properties of the fundamental solution.
The technique applied in  \cite{MR2963461} establishes a connection between the fundamental solution of the difference 
equation and the fundamental solution of an associated differential equation with piecewise continuous arguments,
which allows them to retrieve results to difference equations obtained for delay differential equations.  

Results of this type (i.e., relating certain properties of the fundamental solution with qualitative properties of all solutions)
are of high importance in the theory of delay differential  and difference equations,
as we have an explicit definition of the fundamental solution.
Understanding the nature of the fundamental solution is not only important
in the stability theory but also in the oscillation theory (see \cite{MR2683912})
because it gives information on oscillation of all solutions of the equation.

!!!!!!!!!!!!!!!!!!!!!!!!!!!!!!!

Finally, let us present some open problems and topics for further research.
\begin{enumerate}[label={(P\arabic*)},leftmargin={*},ref=(P\arabic*)]
\item Investigate various types of stability for the dynamic equation with several delays
    \begin{equation}
    x^{\Delta}(t)+\sum_{i=0}^{n}\cf_{i}(t)x(\dly_{i}(t))=0\quad\text{for}\ t\in[t_{0},\infty)_{\T}.\notag
    \end{equation}
    For example, extend the results of \cite{MR0891141,MR1103855,MR0955375} and the present paper.
\item In addition to equations with concentrated delays,
    consider equations with distributed delays, i.e.,
    equations of the form
    \begin{equation}
    x^{\Delta}(t)+\int_{\dly(t)}^{t}K(t,\eta)x(\eta)\Delta\eta=0\quad\text{for}\ t\in[t_{0},\infty)_{\T}.\notag
    \end{equation}
\item Unify the results in \cite[Theorem~4]{MR2814561} and \cite[Theorem~3]{MR1233677} to dynamic equations.
    For instance, show that under \eqref{daleq2} the inequality
    \begin{equation}
    \sup_{t\in[t_{0},\infty)_{\T}}\negmedspace\bigg\{\negmedspace\int_{\dly(t)}^{\sigma(t)}\cf(\eta)\Delta\eta\bigg\}<\frac{3}{2}+\text{some constant}\notag
    \end{equation}
    implies the exponential estimate of Theorem~\ref{dalthm3}{\thinspace}\ref{dalthm3it2}.
    Extend those results to unbounded delays if possible (see, for instance, \cite{MR2437887,MR2335381}).
\item In \cite{MR2914971} it is shown that under certain conditions
    nonoscillation of a dynamic equation is monotonic, i.e.,
    nonoscillation of \eqref{introeq1} on a coarser scale $\T$ implies nonoscillation of the same equation on
a finer time scale $\widetilde{\T}$ satisfying $\widetilde{\T}\supset\T$.
    Is this property preserved for stability or boundedness of solutions on time scales?
\end{enumerate}


\section{Appendix}\label{appx}

\subsection{Appendix A: Continuity of the Fundamental Solution}\label{appa}

\begin{theorem}[Continuity of the fundamental solution]\label{althm1}
The fundamental solution $\fun$ of \eqref{introeq1} is continuous in $\Lambda_{t_{0}}$,
which is defined in \eqref{dalthm1it2eq2}.
\end{theorem}

\begin{proof}
Pick $r\in[t_{0},\infty)_{\T}$, and consider the triangular domain $\Omega_{r}:=\{(t,s):\ t\in[s,r]_{\T}\ \text{and}\ s\in[t_{0},r]_{\T}\}\subset\Lambda_{t_{0}}$.
Note that letting $r\to\sup\T$ implies $\Omega_{r}\to\Lambda_{t_{0}}$.
It is obvious that $\fun(\cdot,s)$ is continuous in $[s,r]_{\T}$ for any fixed $s\in[t_{0},r]_{\T}$.
If we can show that $\fun(t,\cdot)$ is continuous in $[t_{0},t]_{\T}$ uniformly for $t\in[t_{0},r]_{\T}$,
then \cite[Chapter~7, Section~2, Theorem~5]{MR0018708} ensures continuity of $\fun$ in $\Omega_{r}$.
Let $s_{1},s_{2}\in[t_{0},r]_{\T}$, and assume without loss of generality that $s_{2}\geq{}s_{1}$.
Define $y\in\cnt{}([s_{2},r]_{\T},\R_{0}^{+})$ by
\begin{equation}
y(t):=\max_{\eta\in[s_{1},t]_{\T}}|\fun(\eta,s_{2})-\fun(\eta,s_{1})|\quad\text{for}\ t\in[s_{2},r]_{\T}.\notag
\end{equation}
By \cite[Theorem~1.65]{MR1843232}, we may find $M_{1},M_{2}\in\R^{+}$ such that
\begin{equation}
|\fun(t,s_{1})|\leq{}M_{1}\quad\text{for all}\ t\in[\alpha_{\ast}(t_{0}),r]_{\T}\label{althm1prfeq1}
\end{equation}
and
\begin{equation}
|\cf(t)|\leq{}M_{2}\quad\text{for all}\ t\in[t_{0},r]_{\T}.\label{althm1prfeq2}
\end{equation}
By integrating \eqref{introeq1}, we obtain
\begin{equation}
\begin{aligned}[]
& |\fun(t,s_{2})-\fun(t,s_{1})| \\
=&\bigg|\bigg(1-\int_{s_{2}}^{t}\cf(\eta)\fun(\dly(\eta),s_{2})\Delta\eta\bigg)
\phantom{\bigg|}-\bigg(1-\int_{s_{1}}^{t}\cf(\eta)\fun(\dly(\eta),s_{1})\Delta\eta\bigg)\bigg|\\
=&\bigg|\int_{s_{2}}^{t}\cf(\eta)[\fun(\dly(\eta),s_{2})-\fun(\dly(\eta),s_{1})]\Delta\eta
\phantom{\bigg|}-\int_{s_{1}}^{s_{2}}\cf(\eta)\fun(\alpha(\eta),s_{1})\Delta\eta\bigg|\\
\leq&\int_{s_{2}}^{t}|\cf(\eta)||\fun(\dly(\eta),s_{2})-\fun(\dly(\eta),s_{1})|\Delta\eta
+\int_{s_{1}}^{s_{2}}|\cf(\eta)||\fun(\dly(\eta),s_{1})|\Delta\eta
\end{aligned}\notag
\end{equation}
for all $t\in[s_{2},r]_{\T}$.
Now, using \eqref{althm1prfeq1} and \eqref{althm1prfeq2}, we get
\begin{equation}
\begin{aligned}[]
|\fun(t,s_{2})-\fun(t,s_{1})|
\leq&\int_{s_{2}}^{t}|\cf(\eta)|y(\dly(\eta))\chi_{[t_{0},\infty)_{\T}}(\dly(\eta))\Delta\eta+M_{1}M_{2}(s_{2}-s_{1})\\
\leq&\int_{s_{2}}^{t}|\cf(\eta)|y(\eta)\Delta\eta+M_{1}M_{2}(s_{2}-s_{1})\\
\leq&M_{2}\int_{s_{2}}^{t}y(\eta)\Delta\eta+M_{1}M_{2}(s_{2}-s_{1})
\end{aligned}\notag
\end{equation}
for all $t\in[s_{2},r]_{\T}$.
This implies
\begin{equation}
y(t)\leq{}M_{2}\int_{s_{2}}^{t}y(\eta)\Delta\eta+M_{1}M_{2}(s_{2}-s_{1})\quad\text{for all}\ t\in[s_{2},r]_{\T}.\notag
\end{equation}
The application of Gr\"{o}nwall's inequality (see \cite[Theorem~6.4]{MR1843232}) yields
\begin{equation}
y(t)\leq{}M_{1}M_{2}(s_{2}-s_{1})\mathrm{e}_{M_{2}}(t,t_{0})\leq{}M_{1}M_{2}(s_{2}-s_{1})\mathrm{e}_{M_{2}}(r,t_{0})\quad\text{for all}\ t\in[t_{0},r]_{\T}.\notag
\end{equation}
Thus, picking $s_{1}$ and $s_{2}$ sufficiently closer makes $y$ sufficiently small on $[t_{0},r]_{\T}$,
i.e., $\fun(t,\cdot)$ is continuous in $[t_{0},t]_{\T}$ uniformly for $t\in[t_{0},r]_{\T}$.
By \cite[Chapter~7, Section~2, Theorem~5]{MR0018708}, we learn that $\fun$ is continuous in $\Omega_{r}$.
Since $r$ is arbitrary, the proof is complete.
\end{proof}

\begin{remark}
Let $s\in[t_{0},\infty)_{\T}$, then $\lim_{t\to{}s^{+}}\fun(t,s)=\fun(s,s)=1$ while $\lim_{t\to{}s^{-}}\fun(t,s)=0$.
\end{remark}

\subsection{Appendix B: Some Properties of the Exponential Function}\label{appb}

\begin{lemma}[{\protect\cite[Theorem~3.6]{MR3250516}}]\label{allm4}
Assume that $\sup\T=\infty$, $s\in\T$ and $f\in\crd{}(\T,\R_{0}^{+})$,
then the following statements are equivalent:
\begin{multicols}{2}
\begin{enumerate}[label={(\roman*)},leftmargin={*},ref=(\roman*)]
\item\label{allm4it1} $\displaystyle\int_{s}^{\infty}f(\eta)\Delta\eta=\infty$;
\item\label{allm4it2} $\lim\limits_{t\to\infty}\mathrm{e}_{f}(t,s)=\infty$.$\phantom{\displaystyle\int_{s}^{\infty}}$
\end{enumerate}
\end{multicols}
\end{lemma}

We have the following result which relates convergence of improper integrals of $f$ and $\ominus(-f)$,
when $(-f)$ is positively regressive.

\begin{lemma}\label{allm5}
Assume that $\sup\T=\infty$, $s\in\T$ and $f\in\crd{}(\T,\R_{0}^{+})$ with $-f\in\reg{+}(\T,\R)$,
then the following statements are equivalent.
\begin{multicols}{2}
\begin{enumerate}[label={(\roman*)},leftmargin={*},ref=(\roman*)]
\item\label{allm5it1} $\displaystyle\int_{s}^{\infty}f(\eta)\Delta\eta=\infty$; 
\item\label{allm5it2} $\displaystyle\int_{s}^{\infty}\ominus\big(-f(\eta)\big)\Delta\eta=\infty$. 
\end{enumerate}
\end{multicols}
\end{lemma}

\begin{proof}
\ref{allm5it1}$\implies$\ref{allm5it2} This implication is obvious since
\begin{equation}
\ominus\big(-f(t)\big)=\frac{f(t)}{1-\mu(t)f(t)}\geq{}f(t)\quad\text{for all}\ t\in[s,\infty)_{\T}.\notag
\end{equation}
\ref{allm5it2}$\implies$\ref{allm5it1} Now, assume the contrary that \ref{allm5it2} holds but \ref{allm5it1} does not hold,
i.e., $\ell<\infty$, where $\ell:=\int_{s}^{\infty}f(\eta)\Delta\eta$.
This implies $\lim_{t\to\infty}\mu(t)f(t)=0$.
Indeed,
\begin{equation}
\lim_{t\to\infty}\mu(t)f(t)
=\lim_{t\to\infty}\bigg[\int_{s}^{\sigma(t)}f(\eta)\Delta\eta-\int_{s}^{t}f(\eta)\Delta\eta\bigg]
=\ell-\ell=0.\notag
\end{equation}
Thus
\begin{equation}
\lim_{t\to\infty}\frac{\ominus\big(-f(t)\big)}{f(t)}=\lim_{t\to\infty}\frac{1}{1-\mu(t)f(t)}=1,\notag
\end{equation}
which implies by comparison (see \cite[Theorem~4.6, Remark~4.7]{MR1989021}) that $\ell=\infty$,
This is a contradiction, therefore, \ref{allm5it1} holds, which completes the proof. 
\end{proof}

Lemma~\ref{allm4} and Lemma~\ref{allm5} yield the following important properties of the exponential function.

\begin{corollary}\label{alcrl1}
Assume that $\sup\T=\infty$, $s\in\T$ and $f\in\crd{}(\T,\R_{0}^{+})$ with $-f\in\reg{+}(\T,\R)$,
then the following statements are equivalent:
\begin{multicols}{2}
\begin{enumerate}[label={(\roman*)},leftmargin={*},ref=(\roman*)]
\item\label{alcrl1it1} $\displaystyle\int_{s}^{\infty}f(\eta)\Delta\eta=\infty$;$\phantom{\bigg|}$
\item\label{alcrl1it2} $\displaystyle\int_{s}^{\infty}\ominus\big(-f(\eta)\big)\Delta\eta=\infty$;$\phantom{\bigg|}$
\item\label{alcrl1it3} $\lim\limits_{t\to\infty}\mathrm{e}_{f}(t,s)=\infty$;$\phantom{\bigg|}$
\item\label{alcrl1it4} $\lim\limits_{t\to\infty}\mathrm{e}_{\ominus(-f)}(t,s)=\infty$;$\phantom{\bigg|}$
\item\label{alcrl1it5} $\lim\limits_{t\to\infty}\mathrm{e}_{\ominus{}f}(t,s)=0$;$\phantom{\bigg|}$
\item\label{alcrl1it6} $\lim\limits_{t\to\infty}\mathrm{e}_{-f}(t,s)=0$.$\phantom{\bigg|}$
\end{enumerate}
\end{multicols}
\end{corollary}

\subsection{Appendix C: Time Scales Essentials}\label{appc}

A \emph{time scale}, which inherits the standard topology on $\R$, is a nonempty closed subset of reals.
A time scale is denoted by the symbol $\T$, and the intervals with a subscript $\T$ are used to denote the intersection of the usual interval
with $\T$.
For $t\in\T$, we define the \emph{forward jump operator} $\sigma:\T\to\T$ by $\sigma(t):=\inf(t,\infty)_{\T}$, while
the \emph{backward jump operator} $\rho:\T\to\T$ is defined by $\rho(t):=\sup(-\infty,t)_{\T}$, and the \emph{graininess function} $\mu:\T\to\R_{0}^{+}$ is defined to be $\mu(t):=\sigma(t)-t$.
A point $t\in\T$ is called \emph{right-dense} if $\sigma(t)=t$ and/or equivalently $\mu(t)=0$ holds; otherwise, it is called \emph{right-scattered}, and similarly \emph{left-dense} and \emph{left-scattered} points are defined with respect to the backward jump operator.
For $f:\T\to\R$ and $t\in\T$, the $\Delta$-derivative $f^{\Delta}(t)$ of $f$ at the point $t$ is defined to be the number, provided it
exists, with the property that, for any $\varepsilon>0$, there is a neighborhood $U$ of $t$ such that
\begin{equation}
|[f^{\sigma}(t)-f(s)]-f^{\Delta}(t)[\sigma(t)-s]|\leq\varepsilon|\sigma(t)-s|\quad\text{for all}\ s\in U,\notag
\end{equation}
where $f^{\sigma}:=f\circ\sigma$ on $\T$.
We mean the $\Delta$-derivative of a function when we only say derivative unless otherwise is specified.
A function $f$ is called \emph{rd-continuous} provided that it is continuous at right-dense points in
$\T$, and has a finite limit at left-dense points, and the \emph{set of rd-continuous functions}
is denoted by $\crd{}(\T,\R)$.
The set of functions $\crd{1}(\T,\R)$ includes the functions whose derivative is in $\crd{}(\T,\R)$ too.
For a function $f\in\crd{1}(\T,\R)$, the so-called \emph{simple useful formula} holds
\begin{equation}
f^{\sigma}(t)=f(t)+\mu(t)f^{\Delta}(t)\quad\text{for all}\ t\in\T^{\kappa},\notag
\end{equation}
where $\T^{\kappa}:=\T\backslash\{\sup\T\}$ if $\sup\T<\infty$ and satisfies $\rho(\sup\T)<\sup\T$; otherwise, $\T^{\kappa}:=\T$.
For $s,t\in\T$ and a function $f\in\crd{}(\T,\R)$, the $\Delta$-integral of $f$ is defined by
\begin{equation}
\int_{s}^{t}f(\eta)\Delta\eta=F(t)-F(s)\quad\text{for}\ s,t\in\T,\notag
\end{equation}
where $F\in\crd{1}(\T,\R)$ is an antiderivative of $f$, i.e., $F^{\Delta}=f$ on $\T^{\kappa}$.
Table~\ref{ptstbl1} gives the explicit forms of the forward jump, graininess, $\Delta$-derivative
and $\Delta$-integral on the well-known time scales of reals, integers and the quantum set,
respectively.

\begin{table}[ht]
  \centering
    \caption{Forward jump, $\Delta$-derivative, $\Delta$-integral}\label{ptstbl1}
    \begin{tabular}{c|ccc}
    \hline
    $\T$ & $\R$ & $h\Z$ ($h>0$) & $\qZ$ $(q>1)$ \\ \hline
    $\sigma(t)$ & $t$ & $t+h$ & $qt$ \\
    $f^{\Delta}(t)$ & $f^{\prime}(t)$ & $\displaystyle\frac{f(t+h)-f(t)}{h}$ & $\displaystyle\frac{f(qt)-f(t)}{(q-1)t}$ \\
    $\displaystyle\int_{s}^{t}f(\eta)\Delta\eta$ & $\displaystyle\int_{s}^{t}f(\eta){\thinspace}\mathrm{d}\eta$ & $\displaystyle\sum_{\eta=s/h}^{t/h-1}f(h\eta)h$ & $(q-1)\displaystyle\sum_{\eta=\log_{q}(s)}^{\log_{q}(t)-1}f(q^{\eta})q^{\eta}$ \\
    \hline
    \end{tabular}
\end{table}

A function $f\in\crd{}(\T,\R)$ is called \emph{regressive} if $1+\mu f\neq0$ on $\T^{\kappa}$, and \emph{positively regressive} if $1+\mu f>0$ on $\T^{\kappa}$.
The \emph{set of regressive functions} and the \emph{set of positively regressive functions} are denoted by $\reg{}(\T,\R)$ and $\reg{+}(\T,\R)$, respectively, and $\reg{-}(\T,\R)$ is defined similarly.
For simplicity, we denote by $\creg{}(\T,\R)$ the set of real regressive constants, and similarly, we define the sets $\creg{+}(\T,\R)$ and $\creg{-}(\T,\R)$.

Let $f\in\reg{}(\T,\R)$, then the \emph{exponential function} $\mathrm{e}_{f}(\cdot,s)$ on a time scale $\T$ is defined to be the unique solution of the initial value problem
\begin{equation}
\begin{cases}
x^{\Delta}(t)=f(t)x(t)\quad\text{for}\ t\in\T^{\kappa}\\
x(s)=1
\end{cases}\notag
\end{equation}
for some fixed $s\in\T$.

If $f\in\reg{}(\T,\R)$ and $g\in\crd{}(\T,\R)$, then the unique solution of the dynamic equation
\begin{equation}
\begin{cases}
x^{\Delta}(t)=f(t)x(t)+g(t)\quad\text{for}\ t\in\T^{\kappa}\\
x(s)=\alpha
\end{cases}\notag
\end{equation}
is given by
\begin{equation}
x(t)=\mathrm{e}_{f}(t,s)\alpha+\int_{s}^{t}\mathrm{e}_{f}(t,\sigma(\eta))g(\eta)\Delta\eta\quad\text{for}\ t\in\T.\notag
\end{equation}
On the other hand, for $f\in\reg{}(\T,\R)$ and $g\in\crd{}(\T,\R)$, the unique solution of the dynamic equation
\begin{equation}
\begin{cases}
x^{\Delta}(t)=-f(t)x^{\sigma}(t)+g(t)\quad\text{for}\ t\in\T^{\kappa}\\
x(s)=\alpha
\end{cases}\notag
\end{equation}
is
\begin{equation}
x(t)=\mathrm{e}_{\ominus f}(t,s)\alpha+\int_{s}^{t}\mathrm{e}_{\ominus{}f}(t,\eta)g(\eta)\Delta\eta\quad\text{for}\ t\in\T.\notag
\end{equation}

For $h\in\R^{+}$, set $\C_{h}:=\{z\in\C:\ z\neq-1/h\}$, $\Z_{h}:=\{z\in\C:\ -\pi/h<\Img(z)\leq\pi/h\}$, and $\C_{0}:=\Z_{0}:=\C$.
For $h\in\R_{0}^{+}$, we define the \emph{cylinder transformation} $\xi_{h}:\C_{h}\to\Z_{h}$ by
\begin{equation}
\xi_{h}(z):=
\begin{cases}
z,&h=0,\\
\displaystyle\frac{1}{h}\Log(1+hz),&h>0
\end{cases}\quad\text{for}\ z\in\C_{h},\notag
\end{equation}
then the exponential function can also be written in the form
\begin{equation}
\mathrm{e}_{f}(t,s):=\exp\negthickspace\negmedspace\bigg\{\negmedspace\int_{s}^{t}\xi_{\mu(\eta)}\big(f(\eta)\big)\Delta\eta\bigg\}\quad\text{for}\ s,t\in\T.\notag
\end{equation}
Table~\ref{ptstbl2} illustrates the explicit forms of the exponential function on some well-known time scales.

\begin{table}[ht]
  \centering
    \caption{The exponential function}\label{ptstbl2}
    \begin{tabular}{c|ccc}
    \hline
    $\T$ & $\R$ & $h\Z$ ($h>0$) & $\qZ$ $(q>1)$ \\ \hline
    $\mathrm{e}_{f}(t,s)$ & $\exp\negthickspace\bigg\{\negthickspace\displaystyle\int_{s}^{t}f(\eta){\thinspace}\mathrm{d}\eta\bigg\}$ & $\displaystyle\prod_{\eta=s/h}^{t/h-1}\negthickspace\big(1+hf(h\eta)\big)$ & $\displaystyle\prod_{\eta=\log_{q}(s)}^{\log_{q}(t)-1}\negthickspace\big(1+(q-1)q^{\eta}f(q^{\eta})\big)$ \\
    \hline
    \end{tabular}
\end{table}

The exponential function $\mathrm{e}_{f}(\cdot,s)$ is strictly positive on $[s,\infty)_{\T}$ whenever \\ 
$f\in\reg{+}([s,\infty)_{\T},\R)$,
while $\mathrm{e}_{f}(\cdot,s)$ alternates in sign at right-scattered points of the interval $[s,\infty)_{\T}$ provided that $f\in\reg{-}([s,\infty)_{\T},\R)$
For $h\in\R_{0}^{+}$, let $z,w\in\C_{h}$, the \emph{circle plus} $\oplus_{h}$ and the \emph{circle minus} $\ominus_{h}$ are defined by $z\oplus_{h}w:=z+w+hzw$ and $z\ominus_{h}w:=(z-w)/(1+hw)$, respectively.
For $f,g\in\reg{}(\T,\R)$ and $r,s,t\in\T$, the exponential function satisfies the properties $\mathrm{e}_{f}(t,s)\mathrm{e}_{f}(s,r)=\mathrm{e}_{f}(t,r)$, $\mathrm{e}_{f}(t,s)=1/\mathrm{e}_{f}(s,t)=\mathrm{e}_{\ominus_{\mu}f}(s,t)$, $\mathrm{e}_{f}(t,s)\mathrm{e}_{g}(t,s)=\mathrm{e}_{f\oplus_{\mu}g}(t,s)$, $\mathrm{e}_{f}(t,s)/\mathrm{e}_{g}(t,s)=\mathrm{e}_{f\ominus_{\mu}g}(t,s)$.
Throughout the paper, we will abbreviate the operations $\oplus_{\mu}$ and $\ominus_{\mu}$ simply by $\oplus$ and $\ominus$, respectively.
It is also known that $\reg{+}(\T,\R)$ is a subgroup of $\reg{}(\T,\R)$, i.e., $0\in\reg{+}(\T,\R)$, $f,g\in\reg{+}(\T,\R)$ implies $f\oplus_{\mu} g\in\reg{+}(\T,\R)$ and $\ominus_{\mu} f\in\reg{+}(\T,\R)$, where $\ominus_{\mu}f:=0\ominus_{\mu}f$ on $\T$.

The readers are referred to \cite{MR1843232} for further interesting details in the time scale theory.

\subsection*{Acknowledgments}

The authors are grateful to the anonymous reviewer for valuable comments that greatly contributed to the present form of the 
paper.
E.~Braverman was partially supported by the NSERC research grant RGPIN-2015-05976.


\end{document}